\documentclass[11pt,twoside]{article}
\usepackage[utf8]{inputenc}
\usepackage[english]{babel}
\usepackage{fancyhdr}
\usepackage{amssymb,amsmath,amsthm}
\usepackage{graphicx}
\usepackage{mathtools}
\usepackage{enumerate}
\usepackage{algorithm}
\usepackage{algorithmic}
\usepackage{afterpage}
\usepackage{framed}
\usepackage{booktabs}
\usepackage{multirow}
\usepackage{array}
\usepackage{dcolumn}
\usepackage{hhline}
\usepackage{bigstrut}
\usepackage{epstopdf}
\usepackage[inline]{enumitem}
\usepackage{hyperref}
\usepackage{caption}
\usepackage[square,numbers]{natbib}

\newcommand{\R}{\mathbb R} % real numbers
\newcommand{\F}{\mathcal F} % feasible set
\newcommand{\LS}{\mathcal L^0} % level set

\DeclarePairedDelimiter{\norm}{\lVert}{\rVert} % norm
\DeclarePairedDelimiter{\abs}{\lvert}{\rvert} % absolute value
\DeclareMathOperator*{\argmin}{Argmin} % argmin
\DeclareMathOperator*{\argmax}{Argmax} % argmax

\newtheorem{assumption}{Assumption}
\newtheorem{lemma}{Lemma}
\newtheorem{proposition}{Proposition}
\newtheorem{remark}{Remark}
\newtheorem{theorem}{Theorem}
\newtheorem{s_c}{SC (Stepsize Condition)}

\hypersetup{
    colorlinks=false,
    urlcolor=black,
    pdfborder={0 0 0}
}

\newcommand{\minitab}[2][l]{\begin{tabular}{#1}#2\end{tabular}}

\makeatletter
\newcommand{\subalign}[1]{%
  \vcenter{%
    \Let@ \restore@math@cr \default@tag
    \baselineskip\fontdimen10 \scriptfont\tw@
    \advance\baselineskip\fontdimen12 \scriptfont\tw@
    \lineskip\thr@@\fontdimen8 \scriptfont\thr@@
    \lineskiplimit\lineskip
    \ialign{\hfil$\m@th\scriptstyle##$&$\m@th\scriptstyle{}##$\crcr
      #1\crcr
    }%
  }
}
\makeatother

\hypersetup{
    colorlinks=false,
    urlcolor=black,
    pdfborder={0 0 0}
}

\newcommand{\email}[1]{E-mail: \href{mailto:#1}{\texttt{#1}}}

\begin{document}

\sloppy

\thispagestyle{plain}

\setcounter{page}{1}

{\centering
%-----------------------%
% INSERT HERE THE TITLE %
%-----------------------%
\textbf{\LARGE An almost cyclic 2-coordinate descent method for singly linearly constrained problems}

\bigskip\bigskip
%-------------------------%
% INSERT HERE THE AUTHORS %
%-------------------------%
Andrea Cristofari$^*$
\bigskip

}

%----------------------------------%
% INSERT HERE AUTHORS' INFORMATION %
%----------------------------------%
\begin{center}
\small{\noindent$^*$Department of Mathematics, \\
University of Padua, \\
Via Trieste, 63, 35121 Padua, Italy \\
\email{andrea.cristofari@unipd.it}
}
\end{center}

\bigskip\par\bigskip\par
\noindent \textbf{Abstract.}
A block decomposition method is proposed for minimizing a (possibly non-convex) continuously differentiable function
subject to one linear equality constraint and simple bounds on the variables.
The proposed method iteratively selects a pair of coordinates according to an almost cyclic strategy that does not use
first-order information, allowing us not to compute the whole gradient of the objective function during the algorithm.
Using first-order search directions to update each pair of coordinates, global convergence to stationary points is established
for different choices of the stepsize under an appropriate assumption on the level set.
In particular, both inexact and exact line search strategies are analyzed.
Further, linear convergence rate is proved under standard additional assumptions.
Numerical results are finally provided to show the effectiveness of the proposed method.

\bigskip\par
\noindent \textbf{Keywords.} Block coordinate descent methods. Block decomposition methods. Linear convergence rate. SVM.

\bigskip\par
\noindent \textbf{MSC2000 subject classifications.} 90C06. 90C30. 65K05.

\section{Introduction}
Block coordinate descent methods, also known as block decomposition methods, are algorithms that
iteratively update a suitably chosen subset of variables, usually referred to as \textit{working set}, trough an appropriate \textit{optimization step}.
Numerous variants of block coordinate descent methods have been proposed in the literature
that essentially differ from each other in two aspects: the working set selection and the optimization step
(see, e.g.,~\cite{wright:2015} and the references therein for an overview of block coordinate descent methods
in unconstrained optimization).
In the last decades, block coordinate descent methods gained great popularity, especially to address large structured problems,
such as those arising in machine learning, where classical algorithms may not be
so efficient and, sometimes, even not applicable for computational reasons.
Moreover, block coordinate descent methods are well suited for parallelization, allowing to exploit modern computer architectures.

In this paper, we are concerned with the minimization of a continuously differentiable function
subject to one linear equality constraint and simple bounds on the variables.
Many relevant problems can be formulated in this way, such as, e.g., Support Vector Machine training,
the continuous quadratic knapsack problem, resource allocation problems, the page rank problem,
the Chebyshev center problem and the coordination of multi-agent systems.

In the literature, most of the block coordinate descent methods proposed for this class of problems use
gradient information to identify a subset of variables that violate some optimality condition and guarantee, once updated, a certain decrease in the objective
function~\cite{platt:1999,joachims:1999,lin:2001,palagi:2005,fan:2005,lin:2009,beck:2014}.
Other methods select variables in order to satisfy an appropriate descent condition
with a decreasing tolerance~\cite{konnov:2016}, or follow a Gauss-Seidel (cyclic) strategy~\cite{lucidi:2007}.
Different working set selection rules, based on sufficient predicted descent, have also been studied in~\cite{tseng:2010}.
Moreover, a Jacobi-type algorithm has been devised in~\cite{liuzzi:2011} and
a class of parallel decomposition methods for quadratic objective functions has been proposed in~\cite{manno:2018}.

In addition to the above algorithms, different versions of random coordinate descent methods
have been proposed in~\cite{necoara:2013,necoara:2014,reddi:2014,patrascu:2015,necoara:2017},
characterized by the fact that the working set is randomly chosen from a given probability distribution.
From a theoretical point of view, random coordinate descent methods have good convergence properties, given in expectation,
and they turn out to be efficient also in practice.
In particular, since random coordinate descent methods do not use first-order information to choose the working set,
the whole gradient of the objective function does not need to be computed during the iterations,
leading to good performances when the objective function has cheap partial derivatives.

In our method, a working set of two coordinates is iteratively chosen according to the following almost cyclic strategy:
one coordinate is selected in a cyclic manner, while the other one is obtained by considering
the distance of each variable from its nearest bound in some points produced by the algorithm.
We see that this working set selection rule does not use first-order information. So, similarly as in random coordinate descent methods,
the whole gradient of the objective function does not need to be computed during the algorithm
and high efficiency is still achieved when the partial derivatives of the objective function are cheap.
Anyway, differently from random coordinate descent methods, the proposed algorithm has deterministic convergence properties.

More in detail, once a pair of coordinates is selected in the working set, our algorithm
performs a minimization step by moving along a first-order search direction with a certain stepsize.
We first give a general condition on the stepsize that guarantees global convergence to stationary points,
under the assumption that every point of the level set has at least one component strictly between the lower and the upper bound.
Note that this assumption is automatically satisfied in many cases: e.g., when the feasible set is the unit simplex,
or when at least one variable has no bounds.
Then, we describe some practical ways to compute the stepsize, considering different classes of objective functions:
the Armijo line search can be used for general non-convex objective functions, overestimates of the Lipschitz constant
can be used for objective functions with Lipschitz continuous gradient,
and the exact line search can be used when the objective function is strictly convex.

We also show that the proposed method converges linearly under standard additional assumptions.
In particular, two different results are given: asymptotic linear convergence rate is proved when
there are finite bounds on some (or all) of the variables, while non-asymptotic linear convergence rate is proved when there are no bounds on the variables.

Lastly, experimental simulations performed on different classes of test problems show promising results of
the proposed algorithm in comparison with other block coordinate descent methods.

The rest of the paper is organized as follows.
In Section~\ref{sec:prel}, we introduce the notation and recall some preliminary results.
In Section~\ref{sec:alg}, we present the algorithm and carry out the convergence analysis.
In Section~\ref{sec:stepsize}, we describe some practical line search strategies.
In Section~\ref{sec:rate}, we analyze the convergence rate of the algorithm.
In Section~\ref{sec:res}, we show the numerical results.
Finally, we draw some conclusions in Section~\ref{sec:conc}.

\section{Preliminaries and notation}\label{sec:prel}
Let us introduce the notation.
Given a vector $x \in \R^n$, we indicate by $x_i$ the $i$th entry of $x$.
We denote by $e_i \in \R^n$ the vector made of all zeros except for the $i$th entry that is equal to $1$.
Given a function $f \colon \R^n \to \R$, the gradient of $f$ is indicated by $\nabla f$,
the $i$th partial derivative of $f$ is indicated by $\nabla_i f$ and the Hessian matrix of $f$ is indicated by $\nabla^2 f$.
The derivative of a function $f \colon \R \to \R$ is denoted by $\dot f$.
The Euclidean norm of a vector $x \in \R^n$ is indicated by $\norm x$.
%Given a \mbox{scalar $y$} and an integer $k$, if we write $y^k$ the superscript generally refers
%to the iteration counter of the algorithm but, if the apex indicates the power, we use round brackets, i.e., we write $(y)^k$.

Throughout the paper, we focus on the following singly linearly constrained problem with lower and upper bounds on the variables:
\begin{equation}\label{prob}
\begin{split}
& \min \, f(x) \\
& \sum_{i=1}^n x_i = b \\
& l_i \le x_i \le u_i, \quad i = 1,\ldots,n,
\end{split}
\end{equation}
where $f \colon \R^n \to \R$ is a continuously differentiable function, $b \in \R$
and, for all $i = 1,\ldots,n$, we have $l_i < u_i$, with $l_i \in \R \cup \{-\infty\}$ and $u_i \in \R \cup \{\infty\}$.
By slight abuse of standard mathematical notation, we allow variable bounds to be infinite.

Note that every problem of the form
\[
\min \, \{\omega(s) \colon \sum_{i=1}^n a_i s_i = b, \; \bar l_i \le s_i \le \bar u_i, \, i = 1,\ldots,n\}
\]
with $\omega \colon \R^n \to \R$, $b\in \R$ and $a_i \ne 0$, $\bar l_i < \bar u_i$, $i = 1,\ldots,n$, can be rewritten as in~\eqref{prob}
via the following variable transformation: $x_i = a_i s_i$, $i = 1,\ldots,n$,
thus considering the objective function  $f(x) := \omega \bigl((x_1/a_1), \ldots, (x_n/a_n)\bigr)$
and setting the lower and the upper bound on $x$ according to the above transformation.

From now on, we denote the feasible set of problem~\eqref{prob} by $\F$.
Throughout the paper, we assume that $\F \ne \emptyset$.
Moreover, the terms \textit{variable} and \textit{coordinate} will be used interchangeably to indicate each $x_i$, $i=1,\ldots,n$.

Finally, let us recall the following characterization of stationary points of problem~\eqref{prob},
which can be easily derived from KKT conditions.
\begin{proposition}
A feasible point $x^*$ of problem~\eqref{prob} is stationary if and only if there exists $\lambda^* \in \R$ such that,
for all $i = 1,\ldots,n$,
\begin{equation}\label{stat}
\nabla_i f(x^*)
\begin{cases}
\ge \lambda^*, \quad & \text{if } x^*_i = l_i, \\
= \lambda^*,   \quad & \text{if } x^*_i \in (l_i, u_i), \\
\le \lambda^*, \quad & \text{if } x^*_i = u_i.
\end{cases}
\end{equation}
\end{proposition}

\section{The Almost Cyclic 2-Coordinate Descent (AC2CD) method}\label{sec:alg}
In this section, we present the algorithm for solving problem~\eqref{prob}
and we analyze its convergence properties to stationary points.

\subsection{Description of the algorithm}
The proposed Almost Cyclic 2-Coordinate Descent (AC2CD) method is a block decomposition method that
iteratively performs a minimization step with respect to a working set of two coordinates,
chosen by an almost cyclic strategy (note that two is the smallest number of variables that can be updated to maintain feasibility).
A remarkable feature of AC2CD is that the working set selection rule does not use first-order information,
allowing us not to compute the whole gradient of the objective function during the algorithm.

To describe the proposed method, we have to distinguish between \textit{outer iterations}, indicated by the integer $k$,
and \textit{inner iterations}, indicated by the pair of integers $(k,i)$, where $k = 0,1,\ldots$ and $i = 1,\ldots,n$.
Each outer iteration $k$ starts with a feasible point, denoted by $x^k$, and has $n$ inner iterations $(k,1),\ldots,(k,n)$.
Each inner iteration $(k,i)$ starts with a feasible point, denoted by $z^{k,i}$, and produces the successive (feasible) $z^{k,i+1}$
by performing a minimization step with respect to a suitably chosen pair of coordinates.
Outer and inner iterations are linked by the following relation:
\[
z^{k,1} = x^k \quad \text{and} \quad x^{k+1} = z^{k,n+1}, \quad \forall \, k \ge 0.
\]
Namely, each cycle of inner iterations $(k,1),\ldots,(k,n)$ starts from $x^k$ and returns the successive $x^{k+1}$.

To be more specific, given any $x^k$ produced by the algorithm, first we choose
a variable index $j(k) \in \{1,\ldots,n\}$ (as to be described later).
Then, in the cycle of inner iterations $(k,1),\ldots,(k,n)$, we adopt the following rule to choose the working set:
one variable index is selected in a cyclic manner, following an arbitrary order with no repetition,
while the second variable index remains fixed and equal to $j(k)$.
Since only one variable index is selected in a cyclic manner, we name this working set selection rule \textit{almost cyclic}.

From now on, for every inner iteration $(k,i)$ we denote by $p^k_i$ the
variable index of the working set that is selected in a cyclic manner.
So, each minimization step is performed with respect to the two coordinates $z^{k,i}_{p^k_i}$ and $z^{k,i}_{j(k)}$.
In this minimization step, we compute the first-order search direction
$d^{k,i} = [\nabla_{j(k)} f(z^{k,i}) - \nabla_{p^k_i} f(z^{k,i})](e_{p^k_i} - e_{j(k)})$.
Clearly, we have $d^{k,i}_h = 0$ for all $h \in \{1,\ldots,n\} \setminus \{p^k_i,j(k)\}$.
Then, we set
\[
z^{k,i+1} = z^{k,i} + \alpha^{k,i} d^{k,i},
\]
where $\alpha^{k,i} \in \R$ is a feasible stepsize (which will be described later on).

Now, we explain how to choose the index $j(k)$ at the beginning of an outer iteration $k$.
Roughly speaking, $x^k_{j(k)}$ must be ``sufficiently far'' from its nearest bound.
Formally, for each index $h \in \{1,\ldots,n\}$, let us define
the operator $D_h \colon \F \to [0,\infty) \cup \{\infty\}$ that takes as input a feasible point $x$ and returns the distance of $x_h$ from its nearest bound:
\begin{equation}\label{Dh}
D_h(x) := \min\{x_h-l_h,u_h-x_h\}.
\end{equation}
So, for a given $x^k$, we choose $j(k)$ as any index that satisfies
\begin{equation}\label{j(k)}
D_{j(k)} (x^k) \ge \tau D^k,
\end{equation}
where $\tau \in (0,1]$ is a fixed parameter and
\begin{equation}\label{Dk}
D^k := \max_{h=1,\ldots,n} D_h(x^k).
\end{equation}
In other words, the distance between $x^k_{j(k)}$ and its nearest bound must be greater than or equal to a certain fraction (equal to $\tau$)
of the maximum distance between each component of $x^k$ and its nearest bound.

In Algorithm~\ref{alg:ac2cd}, we report the scheme of AC2CD.
Hereinafter, we indicate
\begin{equation}\label{g_d_def}
g^{k,i} := \nabla_{j(k)} f(z^{k,i}) - \nabla_{p^k_i} f(z^{k,i}) \qquad \text{and} \qquad
d^{k,i} := g^{k,i} (e_{p^k_i} - e_{j(k)}),
\end{equation}
as they are defined at steps~6 and~7 of Algorithm~\ref{alg:ac2cd}. Note that
\begin{equation}\label{dir_der}
\nabla f(z^{k,i})^T d^{k,i} = -[\nabla_{j(k)} f(z^{k,i}) - \nabla_{p^k_i} f(z^{k,i})]^2 = -(g^{k,i})^2,
\end{equation}
i.e., $d^{k,i}$ is a descent direction at $z^{k,i}$ if $g^{k,i} \ne 0$
(equivalently, every nonzero $d^{k,i}$ is a descent direction at $z^{k,i}$).

%\floatname{algorithm}{\texttt{\hspace*{0.1truecm}Almost Cyclic 2-Coordinate Descent (AC2CD) method}}
%\renewcommand{\thealgorithm}{}
\begin{algorithm}[h!]
\caption{\texttt{\hspace*{0.1truecm}Almost Cyclic 2-Coordinate Descent (AC2CD) method}}
\label{alg:ac2cd}
\begin{algorithmic}
\par\vspace*{0.1cm}
\item[]\hspace*{-0.1truecm}$\,\,\,0$\vspace{0.1truecm}\hspace*{0.1truecm}
    \textbf{Given} $x^0 \in \F$ and $\tau \in (0,1]$
\item[]\hspace*{-0.1truecm}$\,\,\,1$\vspace{0.1truecm}\hspace*{0.1truecm} \textbf{For} $k = 0,1,\ldots$
\item[]\hspace*{-0.1truecm}$\,\,\,2$\vspace{0.1truecm}\hspace*{0.6truecm} Choose a variable index $j(k) \in \{1,\ldots,n\}$ that satisfies~\eqref{j(k)}
\item[]\hspace*{-0.1truecm}$\,\,\,3$\vspace{0.1truecm}\hspace*{0.6truecm} Choose a permutation $\{p^k_1,\ldots,p^k_n\}$ of $\{1,\ldots,n\}$
\item[]\hspace*{-0.1truecm}$\,\,\,4$\vspace{0.1truecm}\hspace*{0.6truecm} Set $z^{k,1} = x^k$
\item[]\hspace*{-0.1truecm}$\,\,\,5$\vspace{0.1truecm}\hspace*{0.6truecm} \textbf{For} $i = 1,\ldots,n$
\item[]\hspace*{-0.1truecm}$\,\,\,6$\vspace{0.1truecm}\hspace*{1.1truecm} Let $g^{k,i} = \nabla_{j(k)} f(z^{k,i}) - \nabla_{p^k_i} f(z^{k,i})$
\item[]\hspace*{-0.1truecm}$\,\,\,7$\vspace{0.1truecm}\hspace*{1.1truecm} Compute the search direction $d^{k,i} = g^{k,i} (e_{p^k_i} - e_{j(k)})$
\item[]\hspace*{-0.1truecm}$\,\,\,8$\vspace{0.1truecm}\hspace*{1.1truecm} Compute a feasible stepsize $\alpha^{k,i}$ and set $z^{k,i+1} = z^{k,i} + \alpha^{k,i} d^{k,i}$
\item[]\hspace*{-0.1truecm}$\,\,\,9$\vspace{0.1truecm}\hspace*{0.6truecm} \textbf{End for}
\item[]\hspace*{-0.1truecm}$10$\vspace{0.1truecm}\hspace*{0.6truecm} Set $x^{k+1} = z^{k,n+1}$
\item[]\hspace*{-0.1truecm}$11$\vspace{0.1truecm}\hspace*{0.1truecm} \textbf{End for}
%\par\vspace*{0.1cm}
\end{algorithmic}
\end{algorithm}

\begin{remark}\label{rem:inner_it}
For the sake of simplicity, within each outer iteration $k$ of AC2CD we consider $n$ inner iterations,
but they actually are $n-1$, since no pair of coordinates is updated when $p^k_i = j(k)$.
\end{remark}

Now, let us focus on the computation of the stepsize $\alpha^{k,i}$ (step~8 of Algorithm~\ref{alg:ac2cd}).
First, for any inner iterate $z^{k,i}$ we define $\bar \alpha^{k,i}$ as the largest feasible stepsize along the direction $d^{k,i}$.
Since the equality constraint is clearly satisfied by the choice of $d^{k,i}$, by simple calculations we obtain
\begin{equation}\label{max_stepsize_def}
\bar \alpha^{k,i} =
\begin{cases}
(1/g^{k,i}) \, \min \{u_{p^k_i}-z^{k,i}_{p^k_i},z^{k,i}_{j(k)}-l_{j(k)}\}, \quad       & \text{if } g^{k,i} > 0, \\
(1/\abs{g^{k,i}}) \, \min \{z^{k,i}_{p^k_i}-l_{p^k_i},u_{j(k)}-z^{k,i}_{j(k)}\}, \quad & \text{if } g^{k,i} < 0, \\
0,                                                                                     & \text{if } g^{k,i} = 0.
\end{cases}
\end{equation}
In particular, the choice to define $\bar \alpha^{k,i} = 0$ when $g^{k,i} = 0$ is
for convenience of exposition (and without loss of generality),
since it implies that $\bar \alpha^{k,i} = 0$ if $d^{k,i} = 0$ (equivalently, $d^{k,i} \ne 0$ if $\bar \alpha^{k,i} > 0$).

A general rule for the computation of the stepsize $\alpha^{k,i}$ is stated in the following Stepsize Condition~\ref{sc} (SC~\ref{sc}).
As to be shown, this rule can be easily satisfied in practice by different line search strategies
and guarantees global convergence to stationary points under an appropriate assumption on the level set.
%\afterpage{
\begin{framed}
\begin{s_c}\label{sc}
It must hold that
\begin{enumerate}[label=(\roman*), leftmargin=*]
\item $f(z^{k,i+i}) \le f(z^{k,i})$, with $z^{k,i+1} \in \F$, for every inner iteration $(k,i)$; \label{sc_ls}
\item if $\{f(x^k)\}$ converges, then $\displaystyle{\lim_{k \to \infty} \norm{z^{k,i+1}-z^{k,i}} = 0}$, $i = 1,\ldots,n$; \label{sc_lim_dist_inner_iter}
\item for every pair of indices $\hat \imath, \hat \jmath \in \{1,\ldots,n\}$ and every convergent subsequence $\{z^{k,\hat \imath}\}_{K \subseteq \{0,1,\ldots\}}$
    such that $p^k_{\hat \imath}$ is constant for all $k \in K$ and $j(k) = \hat \jmath$ for all $k \in K$,
    if a real number $\xi >0$ exists such that $\displaystyle{\liminf_{\substack{k \to \infty \\ k \in K}} \bar \alpha^{k, \hat \imath} = \xi}$,
    then $\displaystyle{\lim_{\substack{k \to \infty \\ k \in K}} \nabla f(z^{k,\hat \imath})^T d^{k,\hat \imath} = 0}$. \label{sc_lim_gd}
\end{enumerate}
\end{s_c}
\end{framed}
%}
Let us spend a few words on the meaning of SC~\ref{sc}.
Point~\ref{sc_ls} requires that every inner iterate is feasible and does not increase the objective value.
Point~\ref{sc_lim_dist_inner_iter} is a condition usually needed for convergence of block decomposition methods, requiring that the distance of two successive
inner iterates goes to zero.
Finally, point~\ref{sc_lim_gd} requires that the directional derivatives converge to zero over appropriate subsequences
when, for sufficiently large $k$, the largest feasible stepsizes are bounded from below by a positive constant over the same subsequences.

Let us conclude this section by stating the following lemma, which will be useful in the sequel.
It shows that, if points~\ref{sc_ls}--\ref{sc_lim_dist_inner_iter} of SC~\ref{sc} are satisfied,
then every limit point of $\{x^k\}$ is a limit point of $\{z^{k,i}\}$ for every fixed $i = 1,\ldots,n$.

\begin{lemma}\label{lemma:lim_point_z_k_i}
Let $\{x^k\}$ be a sequence of points produced by AC2CD that satisfies points~\ref{sc_ls}--\ref{sc_lim_dist_inner_iter} of SC~\ref{sc},
and let $\{x^k\}_{K \subseteq \{0,1,\ldots\}}$ be a subsequence converging to $\bar x$. Then,
\[
\lim_{\substack{k \to \infty \\ k \in K}} z^{k,i} = \bar x, \quad i = 1,\ldots,n.
\]
\end{lemma}

\begin{proof}
For $i = 1$ the result is trivial, since $z^{k,1} = x^k$ for all $k$.
For any $i \in \{2,\ldots,n\}$, we have $z^{k,i} - z^{k,1} = \sum_{t=1}^{i-1} z^{k,t+1} - z^{k,t}$, and then
\begin{equation}\label{z_k_i-x_k}
\norm{z^{k,i} - x^k} = \norm{z^{k,i} - z^{k,1}} \le \sum_{t=1}^{i-1} \norm{z^{k,t+1} - z^{k,t}}.
\end{equation}
By continuity of $f$, we have $\{f(x^k)\}_K \to f(\bar x)$.
By point~\ref{sc_ls} of SC~\ref{sc}, we also have $f(x^{k+1}) = f(z^{k,n+1}) \le f(z^{k,n}) \le \ldots \le f(z^{k,1}) = f(x^k)$ for all $k \ge 0$,
and then, $\displaystyle{\lim_{k \to \infty} f(x^k) = f(\bar x)}$.
Therefore, using point~\ref{sc_lim_dist_inner_iter} of SC~\ref{sc}, we can write $\displaystyle{\lim_{k \to \infty}\norm{z^{k,t+1} - z^{k,t}}} = 0$
for all $t = 1,\ldots,n$. Combining this relation with~\eqref{z_k_i-x_k}, we get
\begin{equation}\label{lim_point_z}
\lim_{k \to \infty} \norm{z^{k,i} - x^k} = 0.
\end{equation}
Moreover, $\norm{z^{k,i} - \bar x} = \norm{z^{k,i} - x^k + x^k - \bar x} \le \norm{z^{k,i} - x^k} +  \norm{x^k - \bar x}$.
Then, the result is obtained by combining this inequality with~\eqref{lim_point_z} and the fact that $\{x^k\}_K \to \bar x$.
\end{proof}

\subsection{Convergence to stationary points}\label{subsec:conv}
In this subsection, we are concerned with the convergence analysis of AC2CD.

For a given starting point $x^0$, let us first define the level set
\[
\LS := \{x \in \F \colon f(x) \le f(x^0)\}.
\]
To prove global convergence of AC2CD to stationary points, we need the following assumption,
which requires that every point of $\LS$ has at least one component strictly between the lower and the upper bound.
\begin{assumption}\label{assumpt:l0_int_point}
$\forall \, x \in \LS, \, \exists \, i \in \{1,\ldots,n\} \colon x_i \in (l_i,u_i)$.
\end{assumption}

Let us discuss the role played by this assumption, when combined with SC~\ref{sc}, in the convergence analysis of AC2CD.
First, it guarantees that, for every outer iteration $k \ge 0$,
at least one pair of coordinates can be updated if and only if $x^k$ is non-stationary.
Indeed, every $x^k$ remains in $\LS$ by point~\ref{sc_ls} of SC~\ref{sc} and, by the rule used to choose the index $j(k)$,
if Assumption~\ref{assumpt:l0_int_point} holds we have
\[
l_{j(k)} < x^k_{j(k)} < u_{j(k)}, \quad \forall \, k \ge 0.
\]
So, from the stationarity conditions~\eqref{stat}, it is straightforward to verify that a variable index $p^k_i$ exists such that
$\nabla f(z^{k,i})^T d^{k,i} < 0$ with $\bar \alpha^{k,i} > 0$ if and only if $x^k$ is non-stationary,
that is, at least one pair of coordinates can be updated during the inner iterations if and only if $x^k$ is non-stationary.

Second, as to be shown in the proof of Theorem~\ref{th:conv}, if Assumption~\ref{assumpt:l0_int_point} holds
we can prove that $\{\nabla_{j(k)} f(x^k)\}$ converges, over certain subsequences, to the KKT multiplier $\lambda^*$
defined as in~\eqref{stat}. So, for sufficiently large $k$, we can measure the stationarity violation for each coordinate $x^k_h$
by $\nabla_h f(x^k) - \nabla_{j(k)} f(x^k)$ and guarantee that, at the limit, each coordinate satisfies~\eqref{stat}.

Let us also remark that Assumption~\ref{assumpt:l0_int_point} is automatically satisfied in many cases,
for example when $\F$ is the unit simplex, or when at least one variable has no bounds, that is,
if an index $i \in \{1,\ldots,n\}$ exists such that $l_i = -\infty$ and $u_i = \infty$.
Further, in~\cite{lin:2009}, where the same assumption is used, it is shown that, for the Support Vector Machine training problem,
this assumption is satisfied if the smallest eigenvalue of $\nabla^2 f$
is sufficiently large and the starting point $x^0$ is such that $f(x^0)<0$ (see Appendix~B in~\cite{lin:2009} for more details).

Now, we are ready to show that, under Assumption~\ref{assumpt:l0_int_point}, AC2CD converges to stationary points.
\begin{theorem}\label{th:conv}
Let Assumption~\ref{assumpt:l0_int_point} hold and let $\{x^k\}$ be a sequence of points produced by AC2CD that satisfies SC~\ref{sc}.
Then, every limit point of $\{x^k\}$ is stationary for problem~\eqref{prob}.
\end{theorem}

\begin{proof}
Let $x^*$ be a limit point of $\{x^k\}$ and let $\{x^k\}_{K \subseteq \{0,1,\ldots\}}$ be a subsequence converging to $x^*$.
From the instructions of the algorithm, $x^*$ is a feasible point. Moreover, from Lemma~\ref{lemma:lim_point_z_k_i} we can write
\begin{equation}\label{lim_z_k_i}
\lim_{\substack{k \to \infty \\ k \in K}} z^{k,i} = x^*, \quad i = 1,\ldots,n.
\end{equation}
By continuity of $f$, we have that $\{f(x^k)\}_K$ converges to $f(x^*)$.
Since $\{f(x^k)\}$ is monotonically non-increasing (by point~\ref{sc_ls} of SC~\ref{sc})
we have that $\{f(x^k)\}$ converges to $f(x^*)$.
Therefore, by point~\ref{sc_lim_dist_inner_iter} of SC~\ref{sc} it follows that
\begin{equation}\label{lim_dist_inner_iter}
\lim_{k \to \infty} \norm{z^{k,i+1}-z^{k,i}} = 0, \quad i = 1,\ldots,n.
\end{equation}

Using the fact that the set of indices $\{1,\ldots,n\}$ is finite, there exist an index $\hat \jmath \in \{1,\ldots,n\}$
and a further infinite subsequence, that we still denote by $\{x^k\}_K$ without loss of generality, such that
\[
\lim_{\substack{k \to \infty \\ k \in K}} x^k = x^* \quad \text{and} \quad j(k) = \hat \jmath, \qquad \forall \, k \in K.
\]

We first want to show that a real number $\rho > 0$ exists such that
\begin{equation}\label{D_j_k}
\min\{z^{k,i}_{\hat \jmath}-l_{\hat \jmath},u_{\hat \jmath}-z^{k,i}_{\hat \jmath}\} \ge \rho, \quad i = 1,\ldots,n,
\quad \forall \, \text{sufficiently large } k \in K.
\end{equation}
To this extent, by Assumption~\ref{assumpt:l0_int_point} we have that an index $\bar h \in \{1,\ldots,n\}$ exists such that $x^*_{\bar h} \in (l_{\bar h},u_{\bar h})$.
So, using the operator $D_h$ defined as in~\eqref{Dh}, a positive real number $\rho$ exists such that $D_{\bar h}(x^*) \ge (4/\tau) \rho$,
where $\tau \in (0,1]$ is the parameter defined at step~0 of Algorithm~\ref{alg:ac2cd}, used to compute $j(k)$.
Since $\{x^k\}_K \to x^*$, by continuity of $D_h$ we have that $D_{\bar h}(x^k) \ge (2/\tau) \rho$ for all sufficiently large $k \in K$.
Using the definition of $D^k$ given in~\eqref{Dk}, we obtain
\[
D^k \ge (2/\tau) \rho, \quad \forall \, \text{sufficiently large } k \in K.
\]
By the rule used to choose the index $j(k)$, we can write
\[
D_{\hat \jmath}(z^{k,1}) = D_{\hat \jmath}(x^k) \ge \tau D^k \ge 2 \rho, \quad \forall \, \text{sufficiently large } k \in K.
\]
The above relation, combined with~\eqref{lim_dist_inner_iter}, implies that
$D_{\hat \jmath}(z^{k,i}) \ge \rho$, $i = 1,\ldots,n$, for all sufficiently large $k \in K$.
Equivalently,~\eqref{D_j_k} holds.

Now, we want to show that the stationarity conditions~\eqref{stat} are satisfied at $x^*$ with
\begin{equation}\label{lambda_star}
\lambda^* = \nabla_{\hat \jmath} f(x^*).
\end{equation}

Reasoning by contradiction, assume that this is not true.
Then, there exists an index $\hat t \in \{1,\ldots,n\} \setminus \{\hat \jmath\}$ that violates~\eqref{stat}.
Using again the fact that the set of indices $\{1,\ldots,n\}$ is finite,
there also exist an index $\hat \imath \in \{1,\ldots,n\}$
and a further infinite subsequence, that we still denote by $\{x^k\}_K$ without loss of generality,
such that $p^k_{\hat \imath}$ is constant and equal to $\hat t$ for all $k \in K$.
Since also $j(k)$ is constant (and equal to $\hat \jmath$) for all $k \in K$, we thus obtain
\[
p^k_{\hat \imath} = \hat t \quad \text{and} \quad j(k) = \hat \jmath, \qquad \forall \, k \in K.
\]
By~\eqref{lim_z_k_i} and the continuity of $\nabla f$, it follows that $\{\nabla f(z^{k,i})\}_K$ is bounded for all $i = 1,\ldots,n$.
So, a real number $T$ exists such that
\[
\norm{\nabla f(z^{k,i})} \le T, \quad i = 1,\ldots,n, \quad \forall \, k \in K.
\]
Therefore,
\begin{equation}\label{g_k_i_bounded}
\abs{g^{k, \hat \imath}} = \abs{\nabla_{\hat \jmath} f(z^{k, \hat \imath}) - \nabla_{\hat t} f(z^{k, \hat \imath})} \le 2 T, \quad \forall \, k \in K.
\end{equation}
Since we have assumed $\hat t$ to violate~\eqref{stat} with $\lambda^*$ defined as in~\eqref{lambda_star}, one of the following three cases must hold.
\begin{enumerate}[label=(\roman*)]
\item $x^*_{\hat t} \in (l_{\hat t}, u_{\hat t})$ and $\abs{\nabla_{\hat \jmath} f(x^*)-\nabla_{\hat t} f(x^*)} > 0$.
    Taking into account~\eqref{lim_z_k_i}, a real number $\zeta > 0$ exists such that
    \begin{align}
    \min\{z^{k,\hat \imath}_{\hat t} - l_{\hat t}, u_{\hat t} - z^{k,\hat \imath}_{\hat t}\}
    \ge \zeta, \quad & \forall \, \text{sufficiently large } k \in K, \label{proof_contr1} \\
    \abs{g^{k, \hat \imath}} = \abs{\nabla_{\hat \jmath} f(z^{k, \hat \imath}) - \nabla_{\hat t} f(z^{k, \hat \imath})}
    \ge \zeta, \quad & \forall \, \text{sufficiently large } k \in K. \label{proof_contr2}
    \end{align}
    From~\eqref{D_j_k}, \eqref{g_k_i_bounded}, \eqref{proof_contr1}, \eqref{proof_contr2}
    and the definition of $\bar \alpha^{k,i}$ given in~\eqref{max_stepsize_def}, we obtain
    \[
    \bar \alpha^{k,\hat \imath} \ge \min \Bigl\{\frac{\zeta}{2T}, \frac{\rho}{2T}\Bigl\} > 0, \quad \forall \, \text{sufficiently large } k \in K.
    \]
    Therefore, by point~\ref{sc_lim_gd} of SC~\ref{sc} we get
    \[
    0 = \lim_{\substack{k \to \infty \\ k \in K}} \nabla f(z^{k,\hat \imath})^T d^{k,\hat \imath}  =
    \lim_{\substack{k \to \infty \\ k \in K}}
    -\bigl[\nabla_{\hat \jmath} f(z^{k,\hat \imath})-\nabla_{\hat t} f(z^{k,\hat \imath})\bigr]^2,
    \]
    where the second equality follows from~\eqref{dir_der}. We thus obtain a contradiction with~\eqref{proof_contr2}.
\item $x^*_{\hat t} = l_{\hat t}$ and $\nabla_{\hat t} f(x^*) < \nabla_{\hat \jmath} f(x^*)$.
    Taking into account~\eqref{lim_z_k_i}, we have
    \begin{equation}\label{proof_contr3}
    u_{\hat t} - z^{k,\hat \imath}_{\hat t} \ge \frac{u_{\hat t}-l_{\hat t}}2, \quad \forall \, \text{sufficiently large } k \in K,
    \end{equation}
    and a real number $\zeta > 0$ exists such that
    \begin{equation}\label{proof_contr4}
    g^{k, \hat \imath} = \nabla_{\hat \jmath} f(z^{k, \hat \imath}) - \nabla_{\hat t} f(z^{k, \hat \imath})
    \ge \zeta, \quad \forall \, \text{sufficiently large } k \in K.
    \end{equation}
    From~\eqref{D_j_k}, \eqref{g_k_i_bounded}, \eqref{proof_contr3}, \eqref{proof_contr4}
    and the definition of $\bar \alpha^{k,i}$ given in~\eqref{max_stepsize_def}, we obtain
    \[
    \bar \alpha^{k,\hat \imath} \ge \min \Bigl\{\frac{u_{\hat t}-l_{\hat t}}{4T}, \frac{\rho}{2T}\Bigl\} > 0,
    \quad \forall \, \text{sufficiently large } k \in K.
    \]
    Therefore, by point~\ref{sc_lim_gd} of SC~\ref{sc} we get
    \[
    0 = \lim_{\substack{k \to \infty \\ k \in K}} \nabla f(z^{k,\hat \imath})^T d^{k,\hat \imath}  =
    \lim_{\substack{k \to \infty \\ k \in K}} -\bigl[\nabla_{\hat \jmath} f(z^{k,\hat \imath})-\nabla_{\hat t} f(z^{k,\hat \imath})\bigr]^2,
    \]
    where the second equality follows from~\eqref{dir_der}. We thus obtain a contradiction with~\eqref{proof_contr4}.
\item $x^*_{\hat t} = u_{\hat t}$ and $\nabla_{\hat t} f(x^*) > \nabla_{\hat \jmath} f(x^*)$.
    This is a verbatim repetition of the previous case, which again leads to a contradiction.
\end{enumerate}
We can thus conclude that $x^*$ is a stationary point.
\end{proof}

\section{Computation of the stepsize}\label{sec:stepsize}
In this section, we describe some practical ways to compute the stepsize in order to satisfy SC~\ref{sc}.
We consider different classes of objective functions: general non-convex functions,
those with Lipschitz continuous gradient and strictly convex functions.

\subsection{General non-convex objective functions}\label{subsec:armijo}
When the objective function is non-convex, a common way to compute the stepsize is using an inexact line search.
Here, we consider the Armijo line search, which is now described.
Given any inner iterate $z^{k,i}$ and the direction $d^{k,i}$, first we choose a trial feasible stepsize $\Delta^{k,i}$
and then we obtain $\alpha^{k,i}$ by a backtracking procedure. In particular, we set
\begin{equation}\label{armijo_stepsize}
\alpha^{k,i} = (\delta)^c \, \Delta^{k,i},
\end{equation}
where $c$ is the smallest nonnegative integer such that
\begin{equation}\label{armijo_condition}
f (z^{k,i} + \Delta^{k,i} (\delta)^c d^{k,i}) \le f (z^{k,i}) + \gamma \Delta^{k,i} (\delta)^{c\,} \nabla f(z^{k,i})^T d^{k,i}
\end{equation}
and $\delta \in (0,1)$, $\gamma \in (0,1)$ are fixed parameters.

For what concerns the choice of $\Delta^{k,i}$, first it must be feasible, that is, $\Delta^{k,i} \le \bar \alpha^{k,i}$.
Moreover, as to be shown in the proof of the following proposition, $\Delta^{k,i}$ must be bounded from above by a finite positive constant to satisfy point~\ref{sc_lim_dist_inner_iter} of SC~\ref{sc}.
Namely, we require that $\Delta^{k,i} \le A^{k,i}$, where $A^{k,i}$ is any scalar satisfying
$0 < A_l \le A^{k,i} \le A_u < \infty$, with $A_l$ and $A_u$ being fixed parameters.
In conclusion, in the Armijo line search we set $\Delta^{k,i} = \min\{\bar \alpha^{k,i}, A^{k,i}\}$.
In the next proposition, we show that a stepsize computed in this way satisfies SC~\ref{sc}.

\begin{proposition}\label{prop:armijo_stepsize}
Let $\delta \in (0,1)$, $\gamma \in (0,1)$ and $0 < A_l \le A_u < \infty$.
Then, SC~\ref{sc} is satisfied by computing, at every inner iteration $(k,i)$,
the stepsize $\alpha^{k,i}$ as in~\eqref{armijo_stepsize},
where $c$ is the smallest nonnegative integer such that~\eqref{armijo_condition} holds,
$\Delta^{k,i} = \min\{\bar \alpha^{k,i}, A^{k,i}\}$ and $A^{k,i} \in [A_l,A_u]$.
\end{proposition}

\begin{proof}
We first observe that, at every inner iteration $(k,i)$, there exists a nonnegative integer $c$ satisfying~\eqref{armijo_condition},
since, from~\eqref{dir_der}, we have \mbox{$\nabla f(z^{k,i})^T d^{k,i} < 0$} for all $d^{k,i} \ne 0$.

Point~\ref{sc_ls} of SC~\ref{sc} immediately follows from~\eqref{armijo_condition} and the fact that $\alpha^{k,i} \le \Delta^{k,i}$, with $\Delta^{k,i}$ feasible.
Now, we prove point~\ref{sc_lim_dist_inner_iter} of SC~\ref{sc}.
We first show that a real number $\sigma > 0$ exists such that, for all $k \ge 0$, we have
\begin{equation}\label{forcing_funct_armijo}
f(z^{k,i+1}) \le f(z^{k,i}) - \sigma \norm{z^{k,i+1}-z^{k,i}}^2, \quad i = 1,\ldots,n.
\end{equation}
From~\eqref{armijo_stepsize} and~\eqref{armijo_condition}, for all $k \ge 0$ we can write
\begin{equation}\label{arm0}
f(z^{k,i+1}) \le f(z^{k,i}) + \gamma \alpha^{k,i} \nabla f(z^{k,i})^T d^{k,i}, \quad i = 1,\ldots,n.
\end{equation}
Using~\eqref{dir_der}, we obtain
\begin{equation}\label{arm1}
f(z^{k,i+1}) \le f(z^{k,i}) - \gamma \alpha^{k,i} (g^{k,i})^2, \quad i = 1,\ldots,n.
\end{equation}
Moreover, $z^{k,i+1}-z^{k,i} = \alpha^{k,i} g^{k,i} (e_{p^k_i} - e_{j(k)})$, and then, $\norm{z^{k,i+1}-z^{k,i}}^2 = 2 (\alpha^{k,i})^2 (g^{k,i})^2$.
Using this equality in~\eqref{arm1}, and recalling that $\alpha^{k,i} \le A_u < \infty$, we have that
\[
\begin{split}
f(z^{k,i+1}) \le f(z^{k,i}) - \frac{\gamma}{2 A_u} \norm{z^{k,i+1}-z^{k,i}}^2, \quad i = 1,\ldots,n.
\end{split}
\]
Therefore,~\eqref{forcing_funct_armijo} holds with $\sigma = \gamma/(2A_u)$ for all $k \ge 0$.
So, point~\ref{sc_lim_dist_inner_iter} of SC~\ref{sc} is satisfied by combining~\eqref{forcing_funct_armijo} with the fact that
$f(x^{k+1}) = f(z^{k,n+1}) \le f(z^{k,n}) \le \ldots \le f(z^{k,1}) = f(x^k)$ for all $k \ge 0$.

Now, we prove that also point~\ref{sc_lim_gd} of SC~\ref{sc} holds.
Let $\{z^{k,\hat \imath}\}_K$, $\hat \imath$, $\hat \jmath$ and $\bar \alpha^{k,\hat \imath}$
be defined as in point~\ref{sc_lim_gd} of SC~\ref{sc}.
Rearranging the terms in~\eqref{arm0}, we can write
\begin{equation}\label{arm_decr}
f(z^{k,\hat \imath}) - f(z^{k,\hat \imath+1}) \ge \gamma \alpha^{k, \hat \imath} \abs{\nabla f(z^{k,\hat \imath})^T d^{k,\hat \imath}}, \quad \forall \, k \in K.
\end{equation}
Moreover, since $\{z^{k,\hat \imath}\}_K$ converges and $f(x^{k+1}) = f(z^{k,n+1}) \le f(z^{k,n}) \le \ldots \le f(z^{k,1}) = f(x^k)$ for all $k \ge 0$,
by continuity of $f$ we have that $\{f(x^k)\}$ converges, implying that
\begin{equation}\label{lim_f_z_k_i}
\lim_{k \to \infty} [f(z^{k,i+1}) - f(z^{k,i})] = 0, \quad i = 1,\ldots,n.
\end{equation}
Combining~\eqref{arm_decr} and~\eqref{lim_f_z_k_i}, we obtain
\begin{equation}\label{lim_alpha_gd}
\lim_{\substack{k \to \infty \\ k \in K}} \alpha^{k, \hat \imath} \abs{\nabla f(z^{k,\hat \imath})^T d^{k,\hat \imath}} = 0.
\end{equation}
Proceeding by contradiction, we assume that it does not hold that
\begin{equation}\label{lim_gd}
\lim_{\substack{k \to \infty \\ k \in K}} \nabla f(z^{k,\hat \imath})^T d^{k,\hat \imath} = 0.
\end{equation}
Since $\{z^{k,\hat \imath}\}_K$ converges, then it is bounded and, by continuity of $\nabla f$ and the definition of the direction $d^{k,i}$,
also $\bigl\{\nabla f(z^{k,\hat \imath})\bigr\}_K$ and $\{d^{k,\hat \imath}\}_K$ are bounded.
So, if~\eqref{lim_gd} does not hold, there exist further infinite subsequences,
that we still denote by $\{z^{k,\hat \imath}\}_K$, $\bigl\{\nabla f(z^{k,\hat \imath})\bigr\}_K$ and $\{d^{k,\hat \imath}\}_K$
without loss of generality, such that
\begin{equation}\label{lim_z_d}
\lim_{\substack{k \to \infty \\ k \in K}} z^{k, \hat \imath} = \tilde z \in \R^n, \qquad
\lim_{\substack{k \to \infty \\ k \in K}} d^{k, \hat \imath} = \tilde d \in \R^n
\end{equation}
and
\begin{equation}\label{lim_gd_to_eta}
\lim_{\substack{k \to \infty \\ k \in K}} \nabla f(z^{k,\hat \imath})^T d^{k,\hat \imath} = \nabla f(\tilde z)^T \tilde d = -\eta \in \R,
\end{equation}
with $\eta > 0$. From~\eqref{lim_alpha_gd} and~\eqref{lim_gd_to_eta}, we get
\begin{equation}\label{alpha_to_zero}
\lim_{\substack{k \to \infty \\ k \in K}} \alpha^{k,\hat \imath} = 0.
\end{equation}
Since $A^{k,\hat \imath} \ge A_l > 0$ for all $k \ge 0$ and $\bar \alpha^{k,\hat \imath} \ge \xi/2 > 0$
for all sufficiently large $k \in K$, using the definition of $\Delta^{k,\hat \imath}$ we obtain
\[
\Delta^{k,\hat \imath} \ge \min\{\xi/2, A_l\}>0, \quad \forall \, \text{sufficiently large } k \in K.
\]
Consequently, by~\eqref{alpha_to_zero}, an outer iteration $\bar k \in K$ exists such that
\[
\alpha^{k,\hat \imath} < \Delta^{k,\hat \imath}, \quad \forall \, k \ge \bar k, \; k \in K.
\]
The above relation implies that~\eqref{armijo_condition} is satisfied with $c > 0$ for all $k \ge \bar k$, $k \in K$.
Therefore,
\begin{equation}\label{arm_not_sat}
f\biggl(z^{k,\hat \imath} + \frac{\alpha^{k,\hat \imath}}{\delta} d^{k,\hat \imath}\biggr) >
f(z^{k,\hat \imath}) + \gamma \frac{\alpha^{k,\hat \imath}}{\delta} \nabla f(z^{k,\hat \imath})^T d^{k,\hat \imath}, \quad \forall \, k \ge \bar k, \; k \in K.
\end{equation}
By the mean value theorem, we can write
\begin{equation}\label{arm_mean_val_th}
f\biggl(z^{k,\hat \imath} + \frac{\alpha^{k,\hat \imath}}{\delta} d^{k,\hat \imath}\biggr)
= f(z^{k,\hat \imath}) + \frac{\alpha^{k,\hat \imath}}{\delta} \nabla f(\beta^{k,\hat \imath})^T d^{k,\hat \imath},
\end{equation}
where $\displaystyle{\beta^{k,\hat \imath} = z^{k,\hat \imath} + \theta^{k,\hat \imath} \frac{\alpha^{k,\hat \imath}}{\delta} d^{k,\hat \imath}}$
and $\theta^{k,\hat \imath} \in (0,1)$. Using~\eqref{arm_not_sat} and~\eqref{arm_mean_val_th}, we obtain
\begin{equation}\label{arm_contr}
\nabla f(\beta^{k,\hat \imath})^T d^{k,\hat \imath} > \gamma \nabla f(z^{k,\hat \imath})^T d^{k,\hat \imath},
\quad \forall \, k \ge \bar k, \; k \in K.
\end{equation}
Since $\theta^{k,\hat \imath} \in (0,1)$, and taking into account~\eqref{lim_z_d} and~\eqref{alpha_to_zero}, it follows that
$\{\beta^{k,\hat \imath}\}_K \to \tilde z$.
So, passing to the limit in~\eqref{arm_contr}, we have that $\nabla f(\tilde z)^T \tilde d \ge  \gamma \nabla f(\tilde z)^T \tilde d$.
Using~\eqref{lim_gd_to_eta}, we obtain $-\eta \ge -\gamma \eta$,
contradicting the fact that $\eta > 0$ and $\gamma \in (0,1)$. Then, point~\ref{sc_lim_gd} of SC~\ref{sc} holds.
\end{proof}

\subsection{Objective functions with Lipschitz continuous gradient}\label{subsec:lips}
In this subsection, we consider the case where $\nabla f$ is Lipschitz continuous over $\mathcal F$ with constant $L$.
Namely, we assume that
\[
\norm{\nabla f(y) - \nabla f(x)} \le L \norm{y-x}, \quad \forall \, x,y \in \mathcal F.
\]

First, let us recall a result due to Beck~\cite{beck:2014}, which will be useful in the sequel.
\begin{lemma}\label{lemma:beck}
Let us assume that $\nabla f$ is Lipschitz continuous over $\mathcal F$ with constant $L$.
For any point $z \in \mathcal F$ and any pair of indices $i,j \in \{1,\ldots,n\}$, define the function
\[
\phi_{i,j,z}(t) := f\bigl(z+t(e_i-e_j)\bigr), \quad t \in I_{i,j,z},
\]
where $I_{i,j,z}$ is the interval that comprises the feasible stepsizes. Namely,
\begin{equation}\label{feas_interv_lips}
I_{i,j,z} := \{t \in \R \colon z + t (e_i - e_j) \in \F\}.
\end{equation}

Then, every $\dot \phi_{i,j,z}$ is Lipschitz continuous over $I_{i,j,z}$ with constant $L_{i,j} \le 2L$, that is,
\begin{equation}\label{lips_grad}
\abs{\dot \phi_{i,j,z}(t) - \dot \phi_{i,j,z}(s)} \le L_{i,j} \abs{t-s}, \quad \forall \, z \in \mathcal F, \quad \forall \, t,s \in I_{i,j,z}.
\end{equation}
\end{lemma}

\begin{proof}
See Section~3 in~\cite{beck:2014}.
\end{proof}

Now, let $\bar L_{i,j}$ be some positive overestimates of $L_{i,j}$, with $L_{i,j}$ being the Lipschitz constants defined in Lemma~\ref{lemma:beck}. Namely,
\begin{equation}\label{lips_overest}
\bar L_{i,j} \ge L_{i,j}, \; \text{ such that } \; \bar L_{i,j} > 0, \qquad i, j = 1,\ldots,n.
\end{equation}
We will show that these overestimates can be used to compute, in closed form, a stepsize satisfying SC~\ref{sc}.
In particular, for a given fixed parameter $\gamma \in (0,1)$, at every inner iteration $(k,i)$ we can set
\begin{equation}\label{lips_stepsize}
\alpha^{k,i} = \min\biggl\{\bar \alpha^{k,i},\frac{2(1-\gamma)}{\bar L_{p^k_i,j(k)}}\biggr\}.
\end{equation}
As to be pointed out in the proof of the following proposition, this stepsize can be seen as
a particular case of the Armijo stepsize defined in Proposition~\ref{prop:armijo_stepsize}.
Note also that, since $L_{i,j} \le 2L$, every positive overestimate of $2L$ can be used in~\eqref{lips_stepsize}.

\begin{proposition}\label{prop:lips_stepsize}
Let us assume that $\nabla f$ is Lipschitz continuous over $\mathcal F$ with constant $L$ and let $\gamma \in (0,1)$.
Then, SC~\ref{sc} is satisfied by computing, at every inner iteration $(k,i)$,
the stepsize $\alpha^{k,i}$ as in~\eqref{lips_stepsize}.
\end{proposition}

\begin{proof}
Let $i,j \in \{1,\ldots,n\}$ be any pair of indices and consider inequality~\eqref{lips_grad}.
Observing that $0 \in I_{i,j,z}$ for every feasible $z$, and using known results on functions with Lipschitz continuous gradient
(see, e.g.,~\cite{nesterov:2013}), we can write
\[
\phi_{i,j,z}(t) \le \phi_{i,j,z}(0) + t \, \dot \phi_{i,j,z}(0) + \frac{L_{i,j}} 2 t^2,
\quad \forall \, z \in \mathcal F, \quad \forall \, t \in I_{i,j,z}.
\]
Since $\dot \phi_{i,j,z}(t) = \nabla_i f\bigl(z+t(e_i-e_j)\bigr)-\nabla_j f\bigl(z+t(e_i-e_j)\bigr)$, it follows that
\begin{equation}\label{phi_upp}
f\bigl(z+t(e_i-e_j)\bigr) \le f(z) + t [\nabla_i f(z)-\nabla_j f(z)] + \frac{L_{i,j}} 2 t^2,
\quad \forall \, z \in \mathcal F, \quad \forall \, t \in I_{i,j,z}.
\end{equation}

To prove the assertion, it is sufficient to show that $\alpha^{k,i}$, defined as in~\eqref{lips_stepsize},
is a particular case of the Armijo stepsize defined in Proposition~\ref{prop:armijo_stepsize}.
To this extent, we can set $A^{k,i} = 2(1-\gamma)/ \bar L_{p^k_i,j(k)}$
(all these quantities are positive and finite) and, by~\eqref{lips_stepsize}, we obtain
$\alpha^{k,i} = \Delta^{k,i}$, with $\Delta^{k,i}$ defined as in Proposition~\ref{prop:armijo_stepsize}.
So, all we need is to prove that~\eqref{armijo_condition} holds with $c=0$ for this choice of $\Delta^{k,i}$. Namely, we want to show that
\begin{equation}\label{arm_decrease_lips}
f (z^{k,i} + \Delta^{k,i} d^{k,i}) \le f (z^{k,i}) + \gamma \Delta^{k,i} \nabla f(z^{k,i})^T d^{k,i}.
\end{equation}
Since $\Delta^{k,i} \le \bar \alpha^{k,i}$, and taking into account the definition of $d^{k,i}$ given in~\eqref{g_d_def}, we have that
\[
z^{k,i} + \Delta^{k,i} d^{k,i} = z^{k,i} + \Delta^{k,i} g^{k,i} (e_{p^k_i}-e_{j(k)}) \in \mathcal F,
\]
that is, $\Delta^{k,i} g^{k,i} \in I_{p^k_i,j(k),z^{k,i}}$ by definition~\eqref{feas_interv_lips}.
Therefore, using~\eqref{phi_upp} with $i$ and $j$ replaced by $p^k_i$ and $j(k)$, respectively, $z = z^{k,i}$ and $t = \Delta^{k,i} g^{k,i}$, we obtain
\[
\begin{split}
f(z^{k,i} + \Delta^{k,i} d^{k,i}) & \le f(z^{k,i}) - \Delta^{k,i} (g^{k,i})^2
                                    + \frac{\bar L_{p^k_i,j(k)}}2 (\Delta^{k,i})^2(g^{k,i})^2 \\
                                  & = f(z^{k,i}) + \Delta^{k,i} \nabla f(z^{k,i})^T d^{k,i} \biggl(1 - \frac{\bar L_{p^k_i,j(k)}}2 \Delta^{k,i}\biggr),
\end{split}
\]
where the equality follows from~\eqref{dir_der}.
Since $0 \le \Delta^{k,i} \le 2(1-\gamma)/\bar L_{p^k_i,j(k)}$, we get~\eqref{arm_decrease_lips}.
\end{proof}

\begin{remark}\label{rem:lips_stepsize}
As appears from the proof of Proposition~\ref{prop:lips_stepsize},
the stepsize given in~\eqref{lips_stepsize} satisfies SC~\ref{sc} by using, for every pair of indices $i,j \in \{1,\ldots,n\}$,
a positive constant $\bar L_{i,j} \ge L_{i,j}$, where it is sufficient that $L_{i,j}$ satisfies~\eqref{phi_upp}.
In particular, a case of interest is when we have a (possibly non-convex) separable objective function $f(x) = \sum_{i=1}^n f_i(x_i)$
where each $f_i \colon \R \to \R$ has a Lipschitz continuous derivative with constant $L_i$.
In this case, Proposition~\ref{prop:lips_stepsize} holds even with
$\bar L_{i,j}$ replaced by $\bar L_i + \bar L_j$ in~\eqref{lips_stepsize}, where $\bar L_i$ and $\bar L_j$ are
two positive overestimates of $L_i$ and $L_j$, respectively.
This follows from the fact that, in this case,~\eqref{phi_upp} holds even with $L_{i,j}$ replaced by $L_i + L_j$, since,
by Lipschitz continuity of $\dot f_1,\ldots,\dot f_n$, we have
\[
\begin{split}
f\bigl(z + t(e_i-e_j)\bigr) & = f_i (z_i+t) + f_j(z_j-t) + \sum_{h \ne i,j} f_h(z_h) \\
                            & \le f(z) + t [\dot f_i(z_i) - \dot f_j(z_j)] + \frac {L_i + L_j} 2 t^2 \\
                            & = f(z) + t [\nabla_i f(z)-\nabla_j f(z)] + \frac{L_i + L_j} 2 t^2.
\end{split}
\]
\end{remark}

Now, let us analyze the case where the objective function is quadratic of the following form:
$f(x) = \frac 1 2 x^T Q x - q^T x$, with $Q \in \R^{n \times n}$ symmetric and $q \in \R^n$.
Denoting by $Q_{i,j}$ the element of $Q$ in position $(i,j)$, we have (again from~\cite{beck:2014})
\[
L_{i,j} = \abs{Q_{i,i} + Q_{j,j} - 2 Q_{i,j}}, \quad i, j \in \{1,\ldots,n\}.
\]
So, at every inner iteration $(k,i)$, we can easily obtain (a positive overestimate of) $L_{p^k_i,j(k)}$
in order to compute the stepsize $\alpha^{k,i}$ as in~\eqref{lips_stepsize}.

Moreover, if $(d^{k,i})^T Q d^{k,i} > 0$ for a given direction $d^{k,i}$, we can write
\[
0 < \frac 1 {Q_{p^k_i,p^k_i} + Q_{j(k),j(k)} - 2 Q_{p^k_i,j(k)}} =
-\dfrac{\nabla f(z^{k,i})^T d^{k,i}}{(d^{k,i})^T Q d^{k,i}} \in \argmin_{\alpha \in \R} f(z^{k,i} + \alpha d^{k,i}).
\]
It follows that, when $(d^{k,i})^T Q d^{k,i} > 0$, we can set $\gamma = 1/2$, $\bar L_{p^k_i,j(k)} = L_{p^k_i,j(k)}$
and the stepsize given in~\eqref{lips_stepsize} is the exact stepsize, i.e.,
it is the feasible minimizer of $f(z^{k,i} + \alpha d^{k,i})$ with respect to $\alpha$.

Vice versa, if $(d^{k,i})^T Q d^{k,i} \le 0$ for a given direction $d^{k,i}$,
we can exploit the fact that the greatest objective decrease along $d^{k,i}$
is achieved by setting $\alpha^{k,i}$ as large as possible, since
\begin{equation}\label{quad_step}
f(z^{k,i} + \alpha d^{k,i}) = f(z^{k,i}) + \alpha \nabla f(z^{k,i})^T d^{k,i} + \dfrac 1 2 \alpha^2 (d^{k,i})^T Q d^{k,i}, \quad \forall \, \alpha \in \R.
\end{equation}
So, we can set $\alpha^{k,i} = \min\{\bar \alpha^{k,i}, A_u\}$, with $0 < A_u < \infty$ being a (large) fixed parameter.
Using~\eqref{quad_step}, it is easy to see that also this stepsize is a particular case of the Armijo stepsize defined in Proposition~\ref{prop:armijo_stepsize}
(indeed, for every nonnegative value of $\Delta^{k,i}$, the Armijo condition~\eqref{armijo_condition} is satisfied by $c = 0$).

\subsection{Strictly convex objective functions}
Let us consider a strictly convex objective function.
In this case, we can satisfy SC~\ref{sc} by computing, at every inner iteration $(k,i)$, the stepsize $\alpha^{k,i}$ by an exact line search, that is,
\begin{equation}\label{exact_stepsize}
\alpha^{k,i} \in \argmin \, \{f(z^{k,i} + \alpha d^{k,i}) \colon \alpha \in [0,\bar \alpha^{k,i}]\}.
\end{equation}

To ensure that the above minimization is well defined at every inner iteration, we assume that $\LS$ is compact.

\begin{proposition}\label{prop:exact_stepsize}
Let us assume that $f$ is strictly convex and $\LS$ is compact.
Then, SC~\ref{sc} is satisfied by computing, at every inner iteration $(k,i)$,
the stepsize $\alpha^{k,i}$ as in~\eqref{exact_stepsize}.
\end{proposition}

\begin{proof}
Point~\ref{sc_ls} of SC~\ref{sc} immediately follows from the definition of $\alpha^{k,i}$.
Since $f(x^{k+1}) = f(z^{k,n+1}) \le f(z^{k,n}) \le \ldots \le f(z^{k,1}) = f(x^k)$ and each $z^{k,i}$ lies in the compact set $\LS$,
it follows that $\{f(x^k)\}$ converges. So, to prove point~\ref{sc_lim_dist_inner_iter} of SC~\ref{sc}
we have to show that $\displaystyle{\lim_{k \to \infty} \norm{z^{k,i+1}-z^{k,i}} = 0}$, $i = 1,\ldots,n$.
Arguing by contradiction, assume that this is not true.
Then, there exist a real number $\rho > 0$, an index $\hat \imath \in \{1,\ldots,n\}$
and an infinite subsequence $\{z^{k,\hat \imath}\}_{K \subseteq \{0,1,\ldots\}}$ such that
$\norm{z^{k,\hat \imath+1} - z^{k,\hat \imath}} \ge \rho$ for all $k \in K$.
Since every point $z^{k,i}$ lies in the compact set $\LS$, there also exist a further infinite subsequence,
that we still denote by $\{z^{k,\hat \imath}\}_K$ without loss of generality, and two distinct points $z', z'' \in \R^n$ such that
\begin{equation}\label{lim_z}
\lim_{\substack{k \to \infty \\ k \in K}} z^{k,\hat \imath} = z' \quad \text{and} \quad
\lim_{\substack{k \to \infty \\ k \in K}} z^{k,\hat \imath+1} = z''.
\end{equation}
As $\alpha^{k,i}$ is obtained by an exact line search, we can write
\begin{equation}\label{ineq_exact_stepsize}
\begin{split}
f(z^{k,\hat \imath+1}) & \le f\Bigl(z^{k,\hat \imath} + \frac 1 2 \alpha^{k,\hat \imath} d^{k,\hat \imath}\Bigr) = f\Bigl(\frac{z^{k,\hat \imath}+z^{k,\hat \imath+1}}2\Bigr) \\
                       & \le \frac 1 2 f(z^{k,\hat \imath}) + \frac 1 2 f(z^{k,\hat \imath+1}) \le f(z^{k,\hat \imath}),
\end{split}
\end{equation}
where the second inequality follows from the convexity of $f$ and the last inequality follows from the fact that $f(z^{k,\hat \imath+1}) \le f(z^{k,\hat \imath})$.
Moreover, since $\{f(x^k)\}$ converges, a real number $\bar f$ exists such that
\begin{equation}\label{lim_f}
\lim_{k \to \infty} f(z^{k,i}) = \bar f, \quad i = 1,\ldots,n.
\end{equation}
By continuity of the objective function, we can write
\[
\lim_{\substack{k \to \infty \\ k \in K}} f(z^{k,\hat \imath}) = f(z') = \bar f \quad \text{and} \quad
\lim_{\substack{k \to \infty \\ k \in K}} f(z^{k,\hat \imath+1}) = f(z'') = \bar f.
\]
So, passing to the limit in~\eqref{ineq_exact_stepsize}, we obtain
$\bar f \le f\Bigl(\dfrac{z'+z''}2\Bigr) \le \bar f$, that is, $f\Bigl(\dfrac{z'+z''}2\Bigr) = \bar f$.
Adding to the left-hand side of this equality the two quantities $\Bigl(\dfrac 1 2 \bar f - \dfrac 1 2 f(z')\Bigr)$ and
$\Bigl(\dfrac 1 2 \bar f - \dfrac 1 2 f(z'')\Bigr)$, that are both equal to zero, we get
\[
f\Bigl(\dfrac{z'+z''}2\Bigr) = \dfrac 1 2 f(z')  + \dfrac 1 2 f(z''),
\]
contradicting the fact that $f$ is strictly convex and $z' \ne z''$. Then, point~\ref{sc_lim_dist_inner_iter} of SC~\ref{sc} holds.

Finally, point~\ref{sc_lim_gd} of SC~\ref{sc} can be proved by the same arguments used in the proof of
Proposition~\ref{prop:armijo_stepsize} for the Armijo stepsize, just observing that the
objective decrease achieved by the exact line search is greater than or equal to the one achieved by the Armijo line search.
\end{proof}

\section{Convergence rate analysis}\label{sec:rate}
In this section, we show that the convergence rate of AC2CD is linear under standard additional assumptions.

The key to prove linear convergence rate of AC2CD is to show
a relation between the points produced by AC2CD and the points produced by
the classical coordinate descent method applied to an equivalent transformed problem.
Then, linear convergence rate follows from the well known properties of the classical coordinate descent method
proved by Luo and Tseng~\cite{luo:1992} and by Beck and Tetruashvili~\cite{beck:2013}.

In particular, here we give two results.
The first one is more general and states that, eventually, $\{f(x^k)\}$ converges linearly to the optimal value of problem~\eqref{prob}, that is,
$C \in [0,1)$ exists such that $f(x^{k+1})-f(x^*) \le C [f(x^k)-f(x^*)]$ for all sufficiently large $k$, with $x^*$ being the optimal solution of problem~\eqref{prob}.
The second result is for the case where there are no bounds on the variables and establishes a non-asymptotic linear convergence rate,
that is, the above inequality holds for every $k \ge 0$.

Let us start by showing the general result.
First, we need a specific rule to compute the index $j(k)$ in order
to ensure that it remains constant from a certain outer iteration $\hat k$.
In particular, we initialize AC2CD with $\tau \in (0,1)$ (step~0 of Algorithm~\ref{alg:ac2cd})
and, from a certain $k \ge 1$, we adopt the following rule:
$j(k) = j(k-1)$ if this choice satisfies~\eqref{j(k)}, otherwise $j(k)$ is set equal to any index $h$ such that $D_h(x^k) = D^k$,
where $D_h$ and $D^k$ are the operators given in~\eqref{Dh} and~\eqref{Dk}, respectively. Namely,
\begin{equation}\label{rate_j_k}
j(k)
\begin{cases}
= j(k-1), \quad                                           & \text{if this choice satisfies~\eqref{j(k)}}, \\
\displaystyle{\in \argmax_{h=1,\ldots,n} D_h(x^k)}, \quad & \text{otherwise}.
\end{cases}
\end{equation}
We can state the following intermediate lemma.

\begin{lemma}\label{lemma:j_k_rate}
Let Assumption~\ref{assumpt:l0_int_point} hold and let $\tau \in (0,1)$.
Let $\{x^k\}$ be a sequence of points produced by AC2CD,
where $j(k)$ is computed as in~\eqref{rate_j_k} from a certain $k \ge 1$.
Let us also assume that $\lim_{k \to \infty} \{x^k\} = x^* \in \R^n$.

Then, there exist a variable index $\bar \jmath$ and an outer iteration $\hat k$ such that
$j(k) = \bar \jmath$ for all $k \ge \hat k$.
Moreover, $x^*_{\bar \jmath} \in (l_{\bar \jmath}, u_{\bar \jmath})$.
\end{lemma}

\begin{proof}
Let $j^* \in \argmax_{h = 1,\ldots,n} D_h(x^*)$.
From Assumption~\ref{assumpt:l0_int_point}, we have $D_{j^*}(x^*)> 0$.

First, we prove that there exist a variable index $\bar \jmath$ and an outer iteration $\hat k$ such that
$j(k) = \bar \jmath$ for all $k \ge \hat k$.
Arguing by contradiction, assume that this is not true.
Since the set of indices $\{1,\ldots,n\}$ is finite, there exist two indices $j_1,j_2 \in \{1,\ldots,n\}$
and two infinite subsequences $\{x^k\}_{K_1}$ and $\{x^k\}_{K_2}$ such that
\[
\begin{cases}
j(k-1) \ne j_1, \\
j(k) = j_1,
\end{cases}
\forall \, k \in K_1, \qquad \quad \text{and} \qquad \quad
\begin{cases}
j(k-1) = j_1, \\
j(k) = j_2,
\end{cases}
\forall \, k \in K_2.
\]
Since $j(k)$ is computed as in~\eqref{rate_j_k} from a certain $k \ge 1$, it follows that
\begin{alignat*}{2}
& D_{j_1}(x^k) = D^k \ge D_{j^*}(x^k), \qquad         && \forall \, \text{sufficiently large } k \in K_1, \\
& D_{j_1}(x^k) < \tau D^k = \tau D_{j_2}(x^k), \qquad && \forall \, \text{sufficiently large } k \in K_2.
\end{alignat*}
By continuity of the operator $D_h$, we can write
\begin{align*}
& D_{j_1}(x^*) = \lim_{\substack{k \to \infty \\ k \in K_1}} D_{j_1}(x^k) \ge \lim_{\substack{k \to \infty \\ k \in K_1}} D_{j^*}(x^k) = D_{j^*}(x^*), \\
& D_{j_1}(x^*) = \lim_{\substack{k \to \infty \\ k \in K_2}} D_{j_1}(x^k) \le \tau \lim_{\substack{k \to \infty \\ k \in K_2}} D_{j_2}(x^k) = \tau D_{j_2}(x^*).
\end{align*}
Combining these two inequalities, and recalling that $D_{j^*}(x^*) > 0$, we obtain
\[
0 < D_{j^*}(x^*) \le \tau D_{j_2}(x^*).
\]
This is contradiction, since $D_{j^*}(x^*) \ge D_{j_2}(x^*)$ and $\tau \in (0,1)$.
So, a variable index $\bar \jmath$ and an outer iteration $\hat k$ exist such that $j(k) = \bar \jmath$ for all $k \ge \hat k$.

To prove that $x^*_{\bar \jmath} \in (l_{\bar \jmath},u_{\bar \jmath})$, assume by contradiction that $D_{\bar \jmath} (x^*) = 0$.
Since $j(k)$ is computed as in~\eqref{rate_j_k} from a certain $k \ge 1$ and $j(k) = \bar \jmath$ for all $k \ge \hat k$,
for all sufficiently large $k$ we have $D_{\bar \jmath}(x^k) \ge \tau D^k \ge \tau D_{j^*}(x^k)$.
By continuity of the operator $D_h$, we obtain
\[
0 = D_{\bar \jmath}(x^*) = \lim_{k \to \infty} D_{\bar \jmath}(x^k) \ge \tau \lim_{k \to \infty} D_{j^*}(x^k) = \tau D_{j^*}(x^*),
\]
which leads to a contradiction, since $D_{j^*}(x^*) > 0$ and $\tau \in (0,1)$.
Therefore, $x^*_{\bar \jmath} \in (l_{\bar \jmath},u_{\bar \jmath})$.
\end{proof}

Now, we are ready to show that, eventually, $\{f(x^k)\}$ converges linearly to $f(x^*)$ under the following assumption.
\begin{assumption}\label{assumpt:rate}
It holds that
\begin{itemize}
\item $f$ is strictly convex twice continuously differentiable over $\R^n$;
\item the optimal solution of problem~\eqref{prob}, denoted by $x^*$, exists;
\item $\nabla^2 f(x^*) \succ 0$.
\end{itemize}
\end{assumption}

Note that Assumption~\ref{assumpt:rate} implies that $\LS$ is compact (see Lemma~9.1 in~\cite{tseng:1991}).
Therefore, if also Assumption~\ref{assumpt:l0_int_point} holds and SC~\ref{sc} is satisfied,
then $\{x^k\}$ converges to the optimal solution $x^*$.
As a further consequence, under Assumption~\ref{assumpt:rate} we can
compute the stepsize by an exact line search to satisfy SC~\ref{sc} (see Proposition~\ref{prop:exact_stepsize}).

\begin{theorem}\label{th:rate}
Let Assumption~\ref{assumpt:l0_int_point} and~\ref{assumpt:rate} hold, and let $\tau \in (0,1)$.
Let $\{x^k\}$ be a sequence of points produced by AC2CD, where $j(k)$ is computed as in~\eqref{rate_j_k} from a certain $k \ge 1$
and the stepsize is computed as indicated in Proposition~\ref{prop:exact_stepsize}.

Then, a real number \mbox{$C \in [0,1)$} and an outer iteration $\bar k$ exist such that
\[
f(x^{k+1})-f(x^*) \le C [f(x^k)-f(x^*)], \quad \forall \, k \ge \bar k.
\]
\end{theorem}

\begin{proof}
Without loss of generality, we assume that $n>1$ (otherwise the feasible set is either empty or a singleton).
Let $\bar \jmath$ and $\hat k$ be the variable index and the outer iteration defined in Lemma~\ref{lemma:j_k_rate}, respectively.
Namely, $j(k) = \bar \jmath$ for all $k \ge \hat k$.
To simplify the notation, without loss of generality we assume that
\begin{equation}\label{proof_assumpt_j}
p^k_n = \bar \jmath, \quad \forall \, k \ge \hat k, \qquad \qquad \text{and} \qquad \qquad \bar \jmath = n.
\end{equation}
Let us consider the following variable transformation:
\begin{equation}\label{xy}
x_i =
\begin{cases}
y_i,                                             \quad & i \in \{1,\ldots,n\} \setminus \{\bar \jmath\}, \\
b - \displaystyle{\sum_{h \ne \bar \jmath} y_h}, \quad & i = \bar \jmath.
\end{cases}
\end{equation}
Equivalently, we can write
\begin{equation}\label{xy_M}
x = M y + w,
\end{equation}
where, recalling that $\bar \jmath = n$,
\begin{equation}\label{M_w}
M = \begin{bmatrix}
    1      & 0      & \ldots & 0 \\
    0      & 1      & \ldots & 0 \\
    \vdots & \vdots & \ddots & \vdots \\
    0      & 0      & \ldots & 1 \\
    -1     & -1     & \ldots & -1
    \end{bmatrix} \in \R^{n \times (n-1)}
\qquad \text{and} \qquad
w = \begin{bmatrix} 0 \\ 0 \\ \vdots \\ 0 \\ b \end{bmatrix} \in \R^n.
\end{equation}
Note that the columns of $M$ are linearly independent and, for every point $x \in \R^n$ such that $\sum_{i=1}^n x_i = b$,
the linear system $x = M y + w$ has a unique solution given by $y_i = x_i$, $i = 1,\ldots,n-1$.

Now, since $x^*_n \in (l_n,u_n)$ (by Lemma~\ref{lemma:j_k_rate}) and $f$ is strictly convex,
in problem~\eqref{prob} we can remove the bound constraints on the variable $x_n$ and
$x^*$ is still the unique optimal solution.
Consequently, by defining the function $\psi \colon \R^{n-1} \to \R$ as $\psi(y) := f\bigl(M y + w\bigr)$, we can recast problem~\eqref{prob} as follows:
\begin{equation}\label{prob_eq_box}
\begin{split}
& \min \, \psi(y) \\
& l_i \le y_i \le u_i, \quad i = 1,\ldots,n-1.
\end{split}
\end{equation}
We observe that, for every feasible point $y$ of problem~\eqref{prob_eq_box}, the corresponding $x$ obtained by~\eqref{xy_M}
satisfies both the equality constraint $\sum_{i=1}^n x_i = b$ and the bound constraints $l_i \le x_i \le u_i$, $i = 1,\ldots,n-1$
(the bound constraints on $x_n$ are ignored for the reasons explained above).

An optimal solution $y^*$ of problem~\eqref{prob_eq_box} exists, is unique and is given by $y^*_i = x^*_i$, $i = 1,\ldots,n-1$.
Indeed, this $y^*$ clearly satisfies $x^* = M y^* + w$ and $\psi(y^*) = f(x^*)$. So, if a feasible point $\tilde y \ne y^*$ of problem~\eqref{prob_eq_box} existed
such that $\psi(\tilde y) \le \psi(y^*)$, it would mean that $M \tilde y + w \ne M y^* + w$ (since the columns of $M$ are linearly independent)
and $f(M \tilde y + w) \le f(M y^* + w) = f(x^*)$. But this is not possible,
since $M \tilde y$ would be feasible for problem~\eqref{prob} after removing the bound constraints on $x_n$,
and we said above that $x^*$ is still the unique optimal solution of problem~\eqref{prob} even if we remove the bound constraints on $x_n$.

Now, for any $k \ge \hat k$ and for any $i = 1,\ldots,n$, let us define $y^{k,i}$ as the unique vector (feasible for problem~\eqref{prob_eq_box}) satisfying
\begin{equation}\label{zy}
z^{k,i} = M y^{k,i} + w,
\end{equation}
which is given by
\begin{equation}\label{zy2}
y^{k,i}_h = z^{k,i}_h, \quad h = 1,\ldots,n-1.
\end{equation}
For all $k \ge \hat k$, without loss of generality we can consider only the first $n-1$ inner iterations in AC2CD.
Indeed, by~\eqref{proof_assumpt_j} we are assuming that $p^k_n = \bar \jmath$ for all $k \ge \hat k$,
and then, no pair of coordinates is updated in the inner iteration $(k,n)$ for all $k \ge \hat k$ (see Remark~\ref{rem:inner_it}).
Let us consider any inner iteration $(k,i)$, with $k \ge \hat k$ and $i \in \{1,\ldots,n-1\}$.
We want to show that
\begin{equation}\label{z_alpha_y}
z^{k,i} + \alpha d^{k,i} =
M y  \left |_{\subalign{& y_{p^k_i} = y^{k,i}_{p^k_i} - \alpha \nabla_{p^k_i} \psi (y^{k,i}) \\ & y_h = y^{k,i}_h, \, h = 1,\ldots,n-1, \, h \ne p^k_i}} \right.+ w,
\qquad \forall \, \alpha \in \R.
\end{equation}
In other words,~\eqref{z_alpha_y} says that moving $z^{k,i}$ along $d^{k,i}$ with a certain stepsize $\alpha$ corresponds, in the $y$ space, to
moving the $p^k_i$th coordinate of $y^{k,i}$ along $-\nabla_{p^k_i} \psi (y^{k,i})$ with the same stepsize $\alpha$, keeping all the other coordinates of $y^{k,i}$ fixed.
To prove that~\eqref{z_alpha_y} holds, we use again the fact that, for all $\alpha \in \R$, the linear system $z^{k,i} + \alpha d^{k,i} = M y + w$
has a unique solution $y$ given by $y_h = z^{k,i}_h + \alpha d^{k,i}_h$, $h = 1,\ldots,n-1$. Then,~\eqref{z_alpha_y} is equivalent to writing
\begin{subequations}\label{z_alpha_y_2}
\begin{align}
z^{k,i}_{p^k_i} + \alpha d^{k,i}_{p^k_i} & = y^{k,i}_{p^k_i} - \alpha \nabla_{p^k_i} \psi(y^{k,i}), \label{z_alpha_y_2_1} \\
z^{k,i}_h + \alpha d^{k,i}_h             & = y^{k,i}_h, \quad h = 1,\ldots,n-1, \quad h \ne p^k_i. \label{z_alpha_y_2_2}
\end{align}
\end{subequations}
Since $d^{k,i}_h = 0$ for all $h \notin \{p^k_i, \bar \jmath\}$, and we are assuming that $\bar \jmath = n$,~\eqref{z_alpha_y_2_2} immediately follows from~\eqref{zy2}.
To obtain~\eqref{z_alpha_y_2_1}, we use~\eqref{zy2} again to write $y^{k,i}_{p^k_i} = z^{k,i}_{p^k_i}$. Thus, we only need to show that
$d^{k,i}_{p^k_i} = -\nabla_{p^k_i} \psi(y^{k,i})$. This follows from the relation
$\nabla \psi(y) = M^T \nabla f(My + w)$, that, combined with~\eqref{zy} and the definition of $M$ given in~\eqref{M_w}, yields to
\[
\begin{split}
\nabla_{p^k_i} \psi(y^{k,i}) & = \nabla_{p^k_i} f(M y^{k,i} + w) - \nabla_{\bar \jmath} f(M y^{k,i} + w) \\
                             & = \nabla_{p^k_i} f(z^{k,i}) - \nabla_{\bar \jmath} f(z^{k,i}) = -d^{k,i}_{p^k_i}.
\end{split}
\]
Therefore,~\eqref{z_alpha_y_2} holds, implying that~\eqref{z_alpha_y} holds too.
Using $\alpha = \alpha^{k,i}$ in~\eqref{z_alpha_y_2}, for all $k \ge \hat k$ we can write
\begin{align*}
z^{k,i+1}_{p^k_i} & = y^{k,i}_{p^k_i} - \alpha^{k,i} \nabla_{p^k_i} \psi(y^{k,i}), & \quad i = 1,\ldots,n-1, \\
z^{k,i+1}_h       & = y^{k,i}_h, \quad h = 1,\ldots,n-1, \quad h \ne p^k_i,        & \quad i = 1,\ldots,n-1.
\end{align*}
Using~\eqref{zy2} (with $i$ replaced by $i+1$) in the above relations, for all $k \ge \hat k$ we obtain
\begin{subequations}\label{coord_grad}
\begin{align}
y^{k,i+1}_{p^k_i} & = y^{k,i}_{p^k_i} - \alpha^{k,i} \nabla_{p^k_i} \psi(y^{k,i}), & \quad i = 1,\ldots,n-1, \\
y^{k,i+1}_h       & = y^{k,i}_h, \quad h = 1,\ldots,n-1, \quad h \ne p^k_i,        & \quad i = 1,\ldots,n-1.
\end{align}
\end{subequations}
We see that, for all $k \ge \hat k$, the vectors $y^{k,1}, \ldots, y^{k,n}$ are the same that would be generated by
a coordinate descent algorithm applied to problem~\eqref{prob_eq_box}.
In particular, for all $k \ge \hat k$, every $y^{k+1,1}$ is obtained from $y^{k,1}$ by
selecting one coordinate at a time by a Gauss-Seidel (or cyclic) rule and moving it with a certain stepsize.

Now we want to show that, for all sufficiently large $k$, each stepsize $\alpha^{k,i}$ is the same that would be obtained
by performing an exact line search in the updating scheme~\eqref{coord_grad}, applied to problem~\eqref{prob_eq_box}.
First observe that there exists $\tilde k \ge \hat k$ such that $z^{k,i}_{\bar \jmath} \in (l_{\bar \jmath},u_{\bar \jmath})$, $i = 1,\ldots,n$,
for all $k \ge \tilde k$, as $x^*_{\bar \jmath} \in (l_{\bar \jmath},u_{\bar \jmath})$ and $\{z^{k,i}\} \to x^*$, $i = 1,\ldots,n$.
Therefore, since an exact line search is used in AC2CD, for all $i = 1,\ldots,n$ we can write
\begin{equation}\label{proof_rate_exact_stepsize}
\alpha^{k,i} \in \argmin_{\alpha \in \R} \, \{f(z^{k,i} + \alpha d^{k,i}) \, \colon \, z^{k,i}_{p^k_i} + \alpha d^{k,i}_{p^k_i} \in [l_{p^k_i}, u_{p^k_i}]\},
\quad \forall \, k \ge \tilde k.
\end{equation}
In other words, for all $k \ge \tilde k$, the constraints $z^{k,i}_{\bar \jmath} + \alpha d^{k,i}_{\bar \jmath} \in [l_{\bar \jmath}, u_{\bar \jmath}]$
with respect to $\alpha$ are not necessary for the computation of the stepsize in AC2CD.
Combining~\eqref{proof_rate_exact_stepsize} with~\eqref{z_alpha_y}, we get that, for all $k \ge \tilde k$,
each $\alpha^{k,i}$ is the optimal solution of the following one-dimensional problem:
\begin{equation}\label{ls_eq}
\begin{split}
& \min_{\alpha \in \R} \, \psi(y) \\
& y_{p^k_i} = y^{k,i}_{p^k_i} - \alpha \nabla_{p^k_i} \psi(y^{k,i}) \\
& y_h = y^{k,i}_h, \quad h = 1,\ldots,n-1, \quad h \ne p^k_i, \\
& l_{p^k_i} \le y^{k,i}_{p^k_i} - \alpha \nabla_{p^k_i} \psi(y^{k,i}) \le u_{p^k_i}.
\end{split}
\end{equation}
In particular, the last bound constraints in~\eqref{ls_eq} follow from those in~\eqref{proof_rate_exact_stepsize}, using~\eqref{z_alpha_y_2_1}.
Therefore, for all $k \ge \tilde k$, each $\alpha^{k,i}$ is the stepsize that would be obtained
by performing an exact line search in the coordinate descent scheme~\eqref{coord_grad}, applied to problem~\eqref{prob_eq_box}.

So, according to the results established by Luo and Tseng~\cite{luo:1992},
eventually $\{\psi(y^{k,1})\}$ converges linearly to the optimal value of problem~\eqref{prob_eq_box} if the following three conditions hold:
\begin{enumerate*}[label=(\roman*)]
\item $\psi$ is strictly convex twice continuously differentiable over $\R^{n-1}$,
\item the optimal solution $y^*$ of problem~\eqref{prob_eq_box} exists,
\item $\nabla^2 \psi(y^*) \succ 0$.
\end{enumerate*}
The second point has already been proved before, while the other two points follow
by combining Assumption~\ref{assumpt:rate} with the fact that the columns of $M$ are linearly independent and $x^* = My^* + w$.
We conclude that a real number $C \in [0,1)$ and an outer iteration $\bar k \ge \tilde k$ exist such that
\[
\psi(y^{k+1,1})-\psi(y^*) \le C [\psi(y^{k,1})-\psi(y^*)], \quad \forall \, k \ge \bar k.
\]
Since $\psi(y^*) = f(x^*)$, $\psi(y^{k,1}) = f(z^{k,1}) = f(x^k)$ and $\psi(y^{k+1,1}) = f(z^{k+1,1}) = f(x^{k+1})$, we get the result.
\end{proof}

Now, we focus on the case where there are no bounds on the variables, i.e., $l_i = -\infty$ and $u_i = \infty$ for all $i = 1,\ldots,n$.
In this case, a non-asymptotic linear convergence rate can be achieved by using
the stepsize defined in Proposition~\ref{prop:lips_stepsize} (with $\gamma = 1/2$).
To obtain this result, we need the following assumption.

\begin{assumption}\label{assumpt:rate_unbounded}
It holds that
\begin{itemize}
\item $f$ is strongly convex over $\R^n$ with constant $\mu$;
\item $\nabla f$ is Lipschitz continuous over $\R^n$ with constant $L$;
\item the optimal solution of problem~\eqref{prob}, denoted by $x^*$, exists.
\end{itemize}
\end{assumption}

Also in this case, we need to maintain the index $j(k)$ fixed throughout the algorithm.
In particular, now we require $j(k)$ to be equal to any index $\bar \jmath \in \{1,\ldots,n\}$ for all $k \ge 0$
(note that any choice of $\bar \jmath$ is acceptable, since there are no bounds on the variables).

\begin{theorem}\label{th:rate_unbounded}
Let us assume that there are no bounds on the variables in problem~\eqref{prob}, i.e., $l_i = -\infty$ and $u_i = \infty$ for all $i = 1,\ldots,n$,
and let Assumption~\ref{assumpt:rate_unbounded} hold.
Given any index $\bar \jmath \in \{1,\ldots,n\}$, let $\{x^k\}$ be a sequence of points produced by AC2CD, where $j(k) = \bar \jmath$ for all $k \ge 0$
and the stepsize is computed as indicated in Proposition~\ref{prop:lips_stepsize}, with $\gamma = 1/2$.

Then, a real number $C \in [0,1)$ exists such that
\begin{equation}\label{rate_unbounded_1}
f(x^{k+1})-f(x^*) \le C [f(x^k)-f(x^*)], \quad \forall \, k \ge 0.
\end{equation}

In particular, considering the constants $L_{i,j}$ defined in Lemma~\ref{lemma:beck} and the constants $\bar L_{i,j}$ used for the stepsize computation, we have
\begin{equation}\label{rate_unbounded_2}
C = 1 - \frac{\mu}{2 \bar L^{\text{max}}_{\bar \jmath} \Bigl[1 + (n-1) (\sum_{i \ne \bar \jmath} L_{i,\bar \jmath})^2 / (\bar L^{\text{min}}_{\bar \jmath})^2\Bigr]},
\end{equation}
where
$\displaystyle{\bar L^{\text{min}}_{\bar \jmath} := \min_{i \ne \bar \jmath} \bar L_{i,\bar \jmath}}\,$ and
$\displaystyle{\,\bar L^{\text{max}}_{\bar \jmath} := \max_{i \ne \bar \jmath} \bar L_{i,\bar \jmath}}$.
\end{theorem}

\begin{proof}
Following the same line of arguments as in the proof of Theorem~\ref{th:rate},
without loss of generality we assume that $n>1$ and that~\eqref{proof_assumpt_j} holds for all $k \ge 0$.
Then, we consider the variable transformation $x = My + w$ given in~\eqref{xy}, \eqref{xy_M} and~\eqref{M_w}
to define the function $\psi \colon \R^{n-1} \to \R$ as $\psi(y) := f\bigl(M y + w\bigr)$.
We obtain that problem~\eqref{prob} is equivalent to the following problem (which is now unconstrained since there are no bounds on the variables):
\begin{equation}\label{prob_eq_unc}
\min_{y \in \R^{n-1}} \, \psi(y).
\end{equation}
An optimal solution $y^*$ of problem~\eqref{prob_eq_unc} exists, is unique and is given by $y^*_i = x^*_i$, $i = 1,\ldots,n-1$,
satisfying $x^* = M y^* + w$ and $\psi(y^*) = f(x^*)$.
Moreover, for any $k \ge 0$ and for any $i = 1,\ldots,n$, let us define $y^{k,i}$ as in~\mbox{\eqref{zy}--\eqref{zy2}}.
We have that~\eqref{coord_grad} holds for all $k \ge 0$.
Namely, for all $k \ge 0$, every $y^{k+1,1}$ is obtained from $y^{k,1}$ by
selecting one coordinate at a time by a Gauss-Seidel (or cyclic) rule and moving it with a certain stepsize.

Now, let us consider the stepsize $\alpha^{k,i}$.
By hypothesis, it is computed as described in Proposition~\ref{prop:lips_stepsize}, with $\gamma = 1/2$,
using the constants $L_{i,j}$ and $\bar L_{i,j}$ as they are defined in Lemma~\ref{lemma:beck} and in~\eqref{lips_overest}, respectively.
Since there are no bounds on the variables, for every $d^{k,i} \ne 0$ we have $\bar \alpha^{k,i} = \infty$. Then,
\[
\alpha^{k,i} =
\begin{cases}
1 / \bar L_{p^k_i,\bar \jmath} \quad & \text{if } d^{k,i} \ne 0, \\
0,                                   & \text{otherwise}.
\end{cases}
\]
We want to show that $L_{1,\bar \jmath}, \ldots, L_{n-1,\bar \jmath}$ are the coordinatewise Lipschitz constants of $\nabla \psi$ over $\R^{n-1}$.
Namely, for every $i = 1,\ldots,n-1$ and for all $y \in \R^{n-1}$, we want to show that
\begin{equation}\label{block_lips}
\bigl|\nabla_i \psi\bigl(y(\alpha)\bigr) - \nabla_i \psi(y)\bigr| \le L_{i,\bar \jmath} \abs{\alpha}, \quad \forall \, \alpha \in \R,
\end{equation}
where $y(\alpha) \in \R^{n-1}$ is defined as follows:
\[
\begin{split}
y(\alpha)_i & = y_i + \alpha, \\
y(\alpha)_h & = y_h, \quad h = 1,\ldots,n-1, \quad h \ne i.
\end{split}
\]
Recalling that $\bar \jmath = n$ and the definition of $M$ given in~\eqref{M_w}, first observe that
$\nabla_i \psi(y) = \nabla f\bigl(M y + w\bigr)^T (e_i - e_{\bar \jmath})$ and
$\nabla_i \psi\bigl(y(\alpha)\bigr) = \nabla f\bigl(M y(\alpha) + w\bigr)^T (e_i - e_{\bar \jmath}) =
\nabla f\bigl(M y + w + \alpha (e_i - e_{\bar \jmath})\bigr)^T (e_i - e_{\bar \jmath})$.
Thus,
\[
\begin{split}
& \bigl|\nabla_i \psi\bigl(y(\alpha)\bigr) - \nabla_i \psi(y)\bigr| = \\
& \bigl|\nabla f\bigl(M y + w + \alpha (e_i - e_{\bar \jmath})\bigr)^T (e_i - e_{\bar \jmath}) - \nabla f(M y + w)^T (e_i-e_{\bar \jmath})\bigr|.
\end{split}
\]
Considering the functions $\phi_{i,j,z}$ defined in Lemma~\ref{lemma:beck}, we have that
$\dot \phi_{i,j,z}(\alpha) = \nabla f\bigl(z+\alpha(e_i-e_j)\bigr)^T (e_i - e_j)$.
Therefore, from Lemma~\ref{lemma:beck} we get
\[
\bigl|\nabla_i \psi\bigl(y(\alpha)\bigr) - \nabla_i \psi(y)\bigr| =
\abs{\dot \phi_{i,\bar \jmath,My+w}(\alpha) - \dot \phi_{i,\bar \jmath,My+w}(0)} \le L_{i,\bar \jmath} \abs{\alpha}
\]
and then~\eqref{block_lips} holds.

So, for all $k \ge 0$ and all $i=1,\ldots,n-1$ such that $d^{k,i} \ne 0$, each stepsize $\alpha^{k,i}$ is
the reciprocal of an overestimate of $L_{p^k_i, \bar \jmath}$, where $L_{p^k_i, \bar \jmath}$ is
the $p^k_i$th coordinatewise Lipschitz constant of $\nabla \psi$.
According to the results stated by Beck and Tetruashvili~\cite{beck:2013} (see Theorem 3.9 in~\cite{beck:2013}),
$\{\psi(y^{k,1})\}$ has a non-asymptotic linear convergence rate if the following three conditions hold:
\begin{enumerate*}[label=(\roman*)]
\item $\psi$ is strongly convex over $\R^{n-1}$,
\item $\nabla \psi$ is Lipschitz continuous over $\R^{n-1}$,
\item the optimal solution $y^*$ of problem~\eqref{prob_eq_unc} exists.
\end{enumerate*}
The third point has already been proved before.
The first point follows from Assumption~\ref{assumpt:rate} and the fact that the columns of $M$ are linearly independent.
The second point follows from the fact that $L_{1,\bar \jmath}, \ldots, L_{n-1,\bar \jmath}$ are the coordinatewise Lipschitz constants of $\nabla \psi$ over $\R^{n-1}$.
In particular, using known results (see Lemma~2 in~\cite{nesterov:2012}) and the fact that $\bar \jmath = n$, we can write
\begin{equation}\label{lips_psi}
\begin{split}
\norm{\nabla \psi(y') - \nabla \psi(y'')} & \le \sum_{i=1}^{n-1} L_{i,\bar \jmath} \, \norm{y'-y''} \\
                                          & = \sum_{i \ne \bar \jmath} L_{i,\bar \jmath} \, \norm{y'-y''}, \quad \forall \, y',y'' \in \R^{n-1}.
\end{split}
\end{equation}
Therefore, a real number $C \in [0,1)$ exists such that
\begin{equation}\label{rate_psi}
\psi(y^{k+1,1})-\psi(y^*) \le C [\psi(y^{k,1})-\psi(y^*)], \quad \forall \, k \ge 0.
\end{equation}
Then, we get~\eqref{rate_unbounded_1} by using the fact that
$\psi(y^*) = f(x^*)$, $\psi(y^{k,1}) = f(z^{k,1}) = f(x^k)$ and $\psi(y^{k+1,1}) = f(z^{k+1,1}) = f(x^{k+1})$.

Now, we prove~\eqref{rate_unbounded_2}.
First observe that, by Theorem 3.9 in~\cite{beck:2013}, the constant $C$ appearing in~\eqref{rate_psi} is given by
\[
C = 1 - \frac{\mu^{\psi}}{2 \bar L^{\text{max}}_{\bar \jmath} \Bigl[1 + (n-1) (L^{\psi})^2 / (\bar L^{\text{min}}_{\bar \jmath})^2\Bigr]},
\]
where $\mu^{\psi}$ and $L^{\psi}$ are the strong convexity constant of $\psi$ over $\R^{n-1}$
and the Lipschitz constant of $\nabla \psi$ over $\R^{n-1}$, respectively,
while $\bar L^{\text{min}}_{\bar \jmath}$ and $\bar L^{\text{max}}_{\bar \jmath}$ are defined as in the assertion of the theorem.
Using~\eqref{lips_psi}, we can upper bound $L^{\psi}$ with $\sum_{i \ne \bar \jmath} L_{i,\bar \jmath}$.
Therefore, to obtain~\eqref{rate_unbounded_2}, we only have to show that
\begin{equation}\label{proof_strong_conv_constant}
\mu^{\psi} = \mu,
\end{equation}
i.e., we have to show that $\psi$ is strongly convex over $\R^{n-1}$ with constant $\mu$.
Using the fact that $f$ is strongly convex over $\R^n$ with constant $\mu$,
for any $y', y'' \in \R^{n-1}$ and for all $\theta \in [0,1]$ we can write
\[
\begin{split}
\psi(\theta y' + (1-\theta) y'') & = f\bigl(M (\theta y' + (1-\theta) y'') + w\bigr) \\
                                   & = f\bigl(\theta (M y' + w) + (1-\theta) (M y'' + w) \bigr) \\
                                   & \le \theta f (M y' + w) + (1-\theta) f (M y'' + w) + \\
                                   & \quad \, - \frac{\mu}2 \theta (1-\theta) \norm{M y' + w - M y'' - w}^2 \\
                                   & = \theta \psi (y') + (1-\theta) \psi (y'') - \frac{\mu}2 \theta (1-\theta) \norm{M (y'-y'')}^2.
\end{split}
\]
Denoting by $\lambda_{\text{min}}(M^T M)$ the smallest eigenvalue of $M^T M$, we also have
$\norm{M (y'-y'')}^2 = (y'-y'')^T M^T M (y'-y'') \ge \lambda_{\text{min}} (M^T M) \norm{y' - y''}^2$, and then,
\begin{equation}\label{proof_strong_conv}
\psi(\theta y' + (1-\theta) y'') \le \theta \psi (y') + (1-\theta) \psi (y'') - \frac{\mu}2 \lambda_{\text{min}}(M^T M) \theta (1-\theta) \norm{y'-y''}^2.
\end{equation}
Note that $M^T M$ has all entries equal to $1$, except for those on the diagonal that are equal to $2$. Namely,
\[
M^T M = I_{(n-1) \times (n-1)} + 1_{(n-1) \times (n-1)},
\]
where $I_{(n-1) \times (n-1)}$ is the $(n-1)$ dimensional identity matrix and $1_{(n-1) \times (n-1)}$ is the $(n-1) \times (n-1)$ matrix made of all ones.
It follows that $\lambda_{\text{min}}(M^T M) = 1 + 0 = 1$, that, combined with~\eqref{proof_strong_conv},
implies that $\psi$ is strongly convex over $\R^{n-1}$ with constant $\mu$.
Then,~\eqref{proof_strong_conv_constant} holds and the result is obtained.
\end{proof}

Let us conclude this section by discussing the relation between AC2CD and the classical coordinate descent method.
As appears from the proofs of Theorem~\ref{th:rate} and Theorem~\ref{th:rate_unbounded}, this relation is crucial to obtain
linear convergence rate of AC2CD and it also provides further insight into the proposed method.

Indeed, for every $k \ge 0$, a cycle of inner iterations $(k,1), \ldots, (k,n)$ in AC2CD can be seen as a cycle of
iterations of the classical coordinate descent method, with cyclic selection rule, applied to an equivalent transformed problem.
In particular, for every $k \ge 0$ we can use the variable transformation given in~\eqref{xy}, \eqref{xy_M} and~\eqref{M_w}, with $\bar \jmath$ replaced by $j(k)$,
to define the function $\psi(y) := f\bigl(M y + w\bigr)$ and recast~\eqref{prob} as an equivalent problem.
Note that, in absence of any information on $j(k)$, the resulting equivalent transformed problem
\begin{itemize}
\item can change from $k$ to $k+1$, as $j(k)$ can change;
\item has bound constraints $l_i \le y_i \le u_i$, $i \in \{1,\ldots,n\} \setminus \{j(k)\}$, plus the constraints
    $l_{j(k)} \le b - \sum_{i \ne j(k)} y_i \le u_{j(k)}$, where the latter follow from the constraints $l_{j(k)} \le x_{j(k)} \le u_{j(k)}$ in~\eqref{prob}.
\end{itemize}
Then, in the proofs of Theorem~\ref{th:rate} and Theorem~\ref{th:rate_unbounded}, we obtain a linear convergence rate
by exploiting the rule used to compute $j(k)$ that, together with other standard assumptions, guarantees that $j(k) = \bar \jmath$ for all $k \ge \hat k$
and $x^*_{\bar \jmath} \in (l_{\bar \jmath},u_{\bar \jmath})$ (for problems with no bounds on the variables, we have $\hat k = 0$).
In this way, for all sufficiently large $k$, the equivalent transformed problem remains the same and
the constraints $l_{j(k)} \le b - \sum_{i \ne j(k)} y_i \le u_{j(k)}$ can be ignored.
It also follows that quickly identifying the index $\bar \jmath$ in problems with finite bounds on the variables may benefit the algorithm.
Then, it may be useful combining AC2CD with some strategies to
predict which variables are at the bounds in the final solution, such es, e.g., those proposed in~\cite{joachims:1999,chang:2011,raj:2016,cristofari:2018}.

Finally, the approach of transforming~\eqref{prob} into an equivalent simply constrained problem
has some connections with the method proposed by Bertsekas in~\cite{bertsekas:1982} for solving problems with linear inequality constraints.
In particular, at every iteration, the algorithm proposed in~\cite{bertsekas:1982} selects a subset of
$n$ indices that comprises all those corresponding to the binding constraints.
Then, assuming linear independence of these constraints, a linear variable transformation is used to obtain a new problem
with a first block of box constraints, plus a second block of general linear inequality constraints which are not binding at the current point
and can be ignored for the computation of the search direction, obtained by a Newton strategy.
So, the main difference between AC2CD and the method proposed in~\cite{bertsekas:1982} is that,
for every $k \ge 0$, we perform a cycle of inner iterations to update one coordinate at a time in the equivalent transformed problem.

\section{Numerical results}\label{sec:res}
In this section, we present the numerical results of AC2CD on some structured problems where
the computation of the partial derivatives of the objective function is cheap with respect to the computation of the whole gradient.
The codes were written in Matlab (version R2017b) and the experiments were run on an Intel(R) Core(TM) i7-7500U with $16$ GB RAM memory.

First, we considered the Chebyshev center problem and the linear Support Vector Machine (SVM) training problem,
which both can be written as
\begin{equation}\label{convex_quad_prob}
\begin{split}
& \min \, f(x) = \frac 1 2 x^T Q^T Q x - q^T x, \\
& \sum_{i=1}^n x_i = 1 \\
& l_i \le x_i \le u_i, \quad i = 1,\ldots,n,
\end{split}
\end{equation}
for some $Q \in \R^{m \times n}$, $q \in \R^n$, $l_i \in \R \cup \{-\infty\}$ and $u_i \in \R \cup \{\infty\}$, $i = 1,\ldots,n$.
On these convex quadratic problems, we compared AC2CD with the following block decomposition methods:
\begin{itemize}
\item Random Coordinate Descent (RCD)~\cite{necoara:2014}, which, at every iteration, randomly selects two distinct variables
    from a given probability distribution and updates them by minimizing a quadratic model of the objective function;
\item Maximal Violating Pair (MVP)~\cite{platt:1999,fan:2005}, which, at every iteration, selects the two variables that most
    violate the stationarity conditions~\eqref{stat} and moves them by using an appropriate stepsize that, in our case, is computed by an exact line search.
\end{itemize}

Note that RCD, like AC2CD, does not need to compute the whole gradient of the objective function for choosing the working set.
On the other hand, MVP exploits a Gauss-Southwell (or greedy) strategy which requires to calculate the whole vector $\nabla f(x^k)$ at each iteration $k$.

Then, we focused on problems with strongly convex objective function and no bounds on the variables,
for which non-asymptotic linear convergence rate of AC2CD has been proved in Section~\ref{sec:rate}.
In this case, we compared our algorithm with two versions of RCD that were proposed in~\cite{necoara:2017} for problems of the following form:
\begin{equation}\label{unbound_prob}
\begin{split}
& \min \, f(x) = \sum_{i=1}^n f_i(x_i), \\
& \sum_{i=1}^n x_i = 0,
\end{split}
\end{equation}
where each $f_i \colon \R \to \R$ is convex and has a Lipschitz continuous derivative with constant $L_i > 0$.
In our experiments, we considered the same test problems used in~\cite{necoara:2017,xiao:2006},
which are of the form of~\eqref{unbound_prob}, with strongly convex $f_i$, $i = 1,\ldots,n$.

Lastly, we considered the following non-convex problems:
\begin{equation}\label{non_convex_prob}
\begin{split}
& \min \, f(x) = \frac 1 2 x^T Q^T D Q x - q^T x, \\
& \sum_{i=1}^n x_i = 1 \\
& x_i \ge 0, \quad i = 1,\ldots,n,
\end{split}
\end{equation}
for some $Q \in \R^{m \times n}$, $D \in \R^{m \times m}$ and $q \in \R^n$, choosing $D$ as a diagonal matrix so that $Q^T D Q$ is indefinite.
On these problems, AC2CD is still compared with RCD (whose analysis for the non-convex case can be found in~\cite{patrascu:2015})
and MVP with exact line search (whose analysis for non-convex problems over the unit simplex can be found in~\cite{bomze:2018},
just observing that, for this class of problems, MVP coincides with the so called pairwise Frank-Wolfe method).

Finally, to have a fair comparison between AC2CD and RCD, in the following results we consider an outer iteration of RCD
as made of $n$ inner iterations, each of them involving a minimization step with respect to a pair of coordinates.

\subsection{Implementation issues}
In this subsection, we describe some implementation details concerning
\begin{itemize}
\item the computation of $\{p^k_1,\ldots,p^k_n\}$ at step~3 of Algorithm~\ref{alg:ac2cd},
\item the computation of the derivatives of the objective function in each algorithm,
\item the termination criterion used in each algorithm.
\end{itemize}

As regards the first point, a permutation $\{p^k_1,\ldots,p^k_n\}$ of $\{1,\ldots,n\}$
was randomly computed at the beginning of every outer iteration $k$ of AC2CD
(it is known that, in coordinate descent methods with cyclic selection rules,
periodically shuffling the ordering of the variables may speed up convergence in practice~\cite{chang:2008,wright:2015}).

For what concerns the derivatives computation, this operation is straightforward for problem~\eqref{unbound_prob}, given the separable structure of the objective function.
More attention is needed for problems~\eqref{convex_quad_prob} and~\eqref{non_convex_prob}, since,
due to the possibly excessive dimension of $Q$, in these problems the Hessian matrices $Q^T Q$ and $Q^T D Q$ were not stored.
For what concerns problems~\eqref{convex_quad_prob}, using some ideas from~\cite{hsieh:2008,necoara:2014}
we introduced the vector $r(x) := Qx$, updating it during the iterations of each considered algorithm, so that
\begin{equation}\label{part_deriv}
\nabla_i f(x) = {Q_i}^T  r(x) - q_i, \quad i = 1,\ldots,n,
\end{equation}
where $Q_i$ is the $i$th column of $Q$.
We see that computing any partial derivative has a cost $\mathcal O(m)$ and we have an extra cost $\mathcal O(2m)$
to update the vector $r(x)$ (since the considered algorithms can move two variables at a time).
These costs can also reduce when $Q$ is sparse. Similarly, for what concerns problems~\eqref{non_convex_prob}, we introduced the vector $r(x) := DQx$
and~\eqref{part_deriv} still holds (in this case we have a cost $\mathcal O(3m)$ to update $r(x)$, due to the presence of the diagonal matrix $D$).

It is worth observing that, in AC2CD and RCD,
both variables included in the working set at the beginning of an inner iteration may be at the lower or at the upper bound.
In this case, the minimization step is not performed in practice (since the two variables cannot be moved),
and then, the partial derivatives of the objective function do not need to be computed.
So, by inserting a simple check in the scheme of AC2CD and RCD,
we avoided to compute the partial derivatives of the objective function when not necessary.

For what concerns the termination criterion used in AC2CD, let us first rewrite the stationarity conditions~\eqref{stat} in an equivalent form:
\[
\min_{i \colon x_i<u_i} \{\nabla_i f(x)\} - \max_{i \colon x_i>l_i} \{\nabla_i f(x)\} \ge 0.
\]
Exploiting some ideas from~\cite{hsieh:2008}, the termination criterion used in AC2CD tries to approximately satisfy
the above relation without the need of computing the whole gradient of the objective function, in order to preserve efficiency.
To this extent, at the beginning of every outer iteration $k$ we first set $G^k_{\text{min}} = \infty$ and $G^k_{\text{max}} = -\infty$.
Then, at every inner iteration $(k,i)$ we update $G^k_{\text{min}}$ and $G^k_{\text{max}}$ as follows:
for each variable $z^{k,i}_h$ included in the working set, if $z^{k,i}_h < u_h$ then we update $G^k_{\text{min}} \gets \min\{G^k_{\text{min}},\nabla_h f(z^{k,i})\}$,
and if $z^{k,i}_h > l_h$ then we update $G^k_{\text{max}} \gets \max\{G^k_{\text{max}},\nabla_h f(z^{k,i})\}$.
So, AC2CD was stopped at the end of the first outer iteration $k$ satisfying
\begin{equation}\label{stop_cr}
G^k_{\text{min}} - G^k_{\text{max}} \ge -\epsilon,
\end{equation}
where $\epsilon$ is a scalar that was set to $10^{-1}$.
Since $\nabla_{p^k_i} f(z^{k,i})$ and $\nabla_{j(k)} f(z^{k,i})$ are not computed
when both $z^{k,i}_{p^k_i}$ and $z^{k,i}_{j(k)}$ are at the lower or at the upper bound, for the reasons explained above,
we also added a final check after that~\eqref{stop_cr} is satisfied, evaluating all those components of $\nabla f(x^k)$
that were skipped, if any.

For the convex problems considered in our experiments,
we used the final objective value $f^{\text{AC2CD}}$ returned by AC2CD as benchmark for termination of RCD and MVP.
Namely, RCD and MVP were stopped when they produced a point $x^k$ satisfying
\begin{equation}\label{opt_err}
\frac{f(x^k)-f^{\text{AC2CD}}}{1+\abs{f^{\text{AC2CD}}}} \le \nu,
\end{equation}
where $\nu$ is a scalar that was set to $10^{-6}$.
Clearly, a termination criterion based on~\eqref{opt_err} cannot be used when the objective function is non-convex,
since the algorithms may converge to different stationary points.
So, for the non-convex problems~\eqref{non_convex_prob}, the same termination criterion used in AC2CD was also used in RCD
(recall that we consider an outer iteration of RCD as made of $n$ inner iterations),
while MVP, that computes $\nabla f(x^k)$ at each iteration $k$, was stopped when it produced a point $x^k$ such that
$\min_{i=1,\ldots,n} \{\nabla_i f(x^k)\} - \max_{i \colon x_i>0} \{\nabla_i f(x^k)\} \ge -10^{-1}$.

\subsection{Chebyshev center}
Given a finite set of vectors $v^1,\ldots,v^n \in \R^m$, the Chebyshev center problem consists in finding the
smallest ball that encloses all the given points.
It arises in many fields, such as mechanical engineering, biology, environmental science and computer graphics
(see~\cite{xu:2003} and the references therein for more details).
The Chebyshev center problem can be formulated as the following convex standard quadratic problem:
\[
\begin{split}
& \min \, f(x) = \sum_{i=1}^n \sum_{j=1}^n (v^i)^T v^j x_i x_j - \sum_{i=1}^n \norm{v^i}^2  x_i \\
& \sum_{i=1}^n x_i = 1 \\
& x_i \ge 0, \quad i = 1,\ldots,n.
\end{split}
\]

Nine synthetic data sets were created by randomly generating each component of the samples $v^1,\ldots,v^n$ from a standard normal distribution,
using $n = 40000, 60000, 80000$ and $m = 0.01 n, 0.05 n, 0.1 n$ for every fixed $n$.
The nine data sets are listed below:
\begin{enumerate}[label=(\roman*), leftmargin=*, itemsep=-1ex]
\item $n = 40000$, $m = 400$; \label{chebyshev_n_40000_m_400}
\item $n = 40000$, $m = 2000$; \label{chebyshev_n_40000_m_2000}
\item $n = 40000$, $m = 4000$; \label{chebyshev_n_40000_m_4000}
\item $n = 60000$, $m = 600$; \label{chebyshev_n_60000_m_600}
\item $n = 60000$, $m = 3000$; \label{chebyshev_n_60000_m_3000}
\item $n = 60000$, $m = 6000$; \label{chebyshev_n_60000_m_6000}
\item $n = 80000$, $m = 800$; \label{chebyshev_n_80000_m_800}
\item $n = 80000$, $m = 4000$; \label{chebyshev_n_80000_m_4000}
\item $n = 80000$, $m = 8000$. \label{chebyshev_n_80000_m_8000}
\end{enumerate}

For each data set, a randomly chosen vertex of the unit simplex was used as starting point.
In AC2CD, for all $k \ge 1$ the index $j(k)$ was computed as in~\eqref{rate_j_k}, with $\tau = 0.9$
(for $k = 0$ we set $j(k) \in \argmax_{h=1,\ldots,n} D_h(x^k)$),
and the stepsize was computed as described in the last part of Subsection~\ref{subsec:lips} for quadratic objective functions,
with $\gamma = 1/2$ and $A_u = 10^{12}$.
In RCD, an $\mathcal O(1)$ procedure was used at each inner iteration to randomly choose, from a uniform distribution,
the pair of distinct variables to be updated.
In particular, this procedure first randomly generates a real number $r$ from a uniform distribution on $(0,\frac{n(n-1)}2)$, and then sets
$i = 1 + \bigl\lfloor (\sqrt{1+8\lfloor r \rfloor}+1)/2 \bigr\rfloor$ and $j = 1 + \bigl\lfloor \lfloor r \rfloor - (i-2)(i-1)/2 \bigr\rfloor$,
where $\lfloor \cdot \rfloor$ denotes the floor operation.

In Table~\ref{tab:chebyshev}, we report the results for each algorithm in terms of final objective value, number of (outer) iterations and CPU time in seconds.
To analyze how fast the objective function decreases in the three considered algorithms,
in Figure~\ref{fig:chebyshev} we plot the normalized optimization error $E^k$ versus the CPU time.
Coherently with the termination criterion used in the experiments, $E^k$ is computed as the left-hand side of~\eqref{opt_err}.

{\setlength{\tabcolsep}{0.53em}
\begin{table}
\centering
\caption{Results of AC2CD, RCD and MVP on Chebyshev center problems.
For each algorithm, the first column indicates the final objective value, the second column indicates the number of (outer) iterations
and the third column indicates the CPU time in seconds.}
{\begin{tabular}{ c | c c c | c c c | c c c}
\hline
\multirow{3}*{\minitab[c]{\textbf{Data} \\ \textbf{set}}}
& \multicolumn{3}{c|}{\textbf{AC2CD}} & \multicolumn{3}{c|}{\textbf{RCD}}  & \multicolumn{3}{c}{\textbf{MVP}} \bigstrut[t] \\
& \multirow{2}*{Obj} & Outer & Time   & \multirow{2}*{Obj} & Outer & Time  & \multirow{2}*{Obj} & \multirow{2}*{Iter} & Time \bigstrut[t] \\
&                    & iter  & (s)    &                    & iter  & (s)   &                    &                     & (s) \bigstrut[b] \\
\hline
\ref{chebyshev_n_40000_m_400} & $-492.77$ & $27$ & $1.43$ & $-492.77$ & $4520$ & $15.26$ & $-492.77$ & $183$ & $0.98$ \bigstrut[t] \\
\ref{chebyshev_n_40000_m_2000} & $-2191.46$ & $26$ & $2.65$ & $-2191.46$ & $2103$ & $11.95$ & $-2191.46$ & $218$ & $4.95$ \\
\ref{chebyshev_n_40000_m_4000} & $-4254.52$ & $25$ & $8.02$ & $-4254.52$ & $1501$ & $24.78$ & $-4254.52$ & $281$ & $12.62$ \\
\ref{chebyshev_n_60000_m_600} & $-716.28$ & $31$ & $2.69$ & $-716.28$ & $6160$ & $31.31$ & $-716.28$ & $168$ & $1.92$ \\
\ref{chebyshev_n_60000_m_3000} & $-3228.41$ & $26$ & $7.19$ & $-3228.40$ & $2225$ & $27.45$ & $-3228.40$ & $308$ & $15.48$ \\
\ref{chebyshev_n_60000_m_6000} & $-6314.02$ & $24$ & $13.59$ & $-6314.01$ & $1602$ & $40.61$ & $-6314.01$ & $365$ & $36.42$ \\
\ref{chebyshev_n_80000_m_800} & $-933.28$ & $30$ & $3.86$ & $-933.28$ & $7391$ & $52.48$ & $-933.28$ & $202$ & $3.73$ \\
\ref{chebyshev_n_80000_m_4000} & $-4267.89$ & $27$ & $16.83$ & $-4267.89$ & $2563$ & $58.99$ & $-4267.89$ & $368$ & $32.50$ \\
\ref{chebyshev_n_80000_m_8000} & $-8370.25$ & $25$ & $21.00$ & $-8370.24$ & $1743$ & $58.22$ & $-8370.24$ & $421$ & $75.95$ \bigstrut[b] \\
\hline
\end{tabular}}
\label{tab:chebyshev}
\end{table}
}

\begin{figure}
\centering
\includegraphics[trim = 0.4cm 0.9cm 0.4cm 0.7cm, clip]{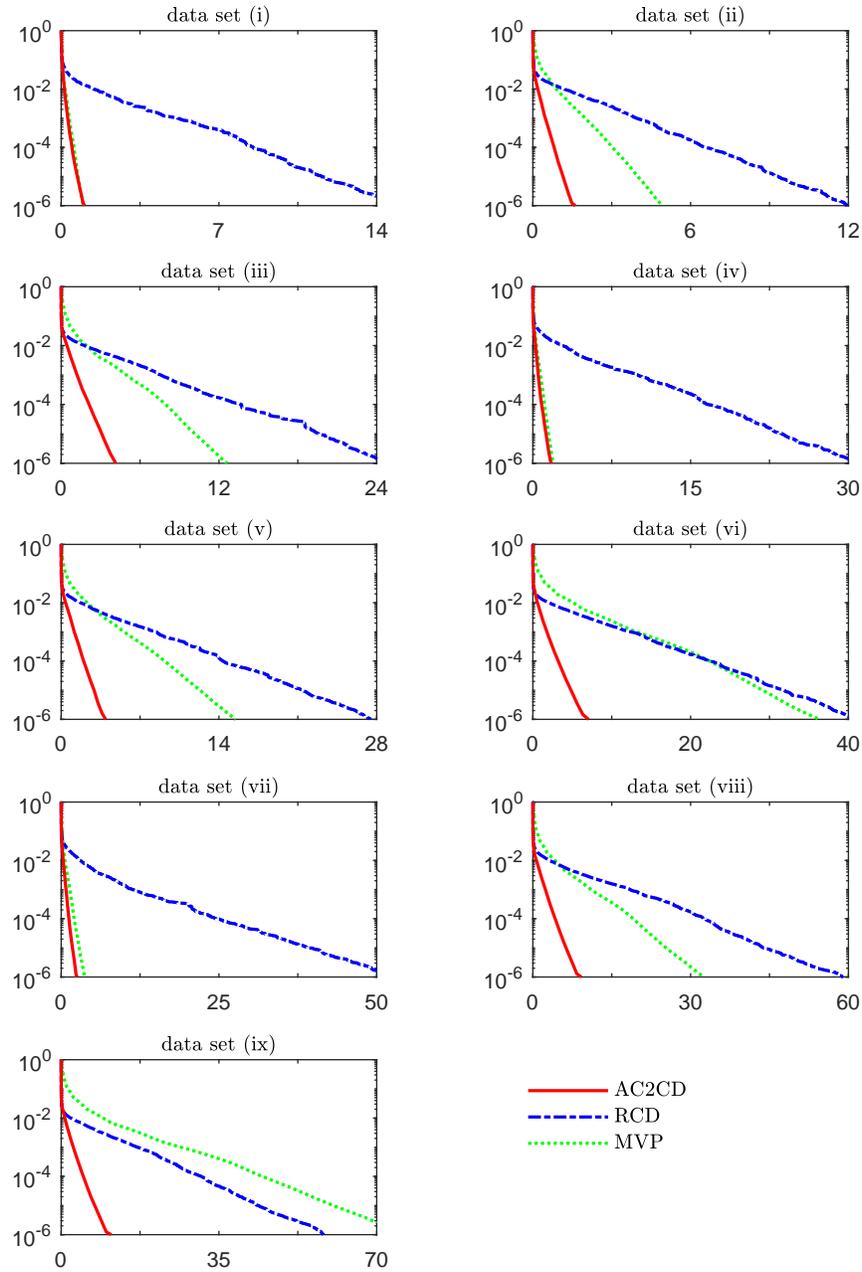}
\caption{Normalized optimization error ($y$ axis) versus CPU time in seconds ($x$ axis) for AC2CD, RCD and MVP on Chebyshev center problems.
The $y$ axis is in logarithmic scale and, for each algorithm, the normalized optimization error is computed as the left-hand side of~\eqref{opt_err}.}
\label{fig:chebyshev}
\end{figure}

We see that, on all the considered data sets, AC2CD outperforms RCD both in CPU time, by a factor between $2$ and $14$,
and in the number of outer iterations, by a factor between $60$ and $247$.
Looking at Figure~\ref{fig:chebyshev} more in detail, we observe that AC2CD and RCD are comparable in the first iterations,
but AC2CD is able to compute a solution with higher precision in a smaller amount of time.
In comparison with MVP, we observe that AC2CD is slightly slower on the data sets with $m = 0.01n$, i.e.,
data sets~\ref{chebyshev_n_40000_m_400}, \ref{chebyshev_n_60000_m_600} and~\ref{chebyshev_n_80000_m_800}.
On all the other data sets, on average AC2CD is more than $2$ times faster than MVP in CPU time.

\subsection{Linear SVM}
Linear Support Vector Machine~\cite{boser:1992} is a popular technique for data classification,
which aims at separating a given set of samples by a hyperplane.
Formally, let $v^1,\ldots,v^n \in \R^m$ be a finite set of vectors and $a^1,\ldots,a^n \in \{-1,+1\}$ be the corresponding labels.
To train a linear SVM, we can solve the following convex quadratic problem:
\[
\begin{split}
& \min \, f(x) = \frac 1 2 \sum_{i=1}^n \sum_{j=1}^n a^i a^j (v^i)^T v^j x_i x_j - \sum_{i=1}^n x_i \\
& \sum_{i=1}^n a^i x_i = 0 \\
& 0 \le x_i \le C, \quad i = 1,\ldots,n,
\end{split}
\]
where $C$ is a positive parameter, set to $1$ in our experiments.
As mentioned in Section~\ref{sec:prel}, the above problem can be easily rewritten as in~\eqref{prob}.

Eight data sets were downloaded from the LIBSVM~\cite{chang:2011} webpage \url{https://www.csie.ntu.edu.tw/~cjlin/libsvmtools/datasets}.
They are listed below:
\begin{enumerate}[label=(\roman*), leftmargin=*, itemsep=-1ex]
\item gisette ($n = 6000$, $m = 5000$); \label{svm_dataset_1}
\item rcv1 ($n = 20242$, $m = 47236$); \label{svm_dataset_2}
\item a9a ($n = 32561$, $m = 123$); \label{svm_dataset_3}
\item w8a ($n = 49749$, $m = 300$); \label{svm_dataset_4}
\item ijcnn1 ($n = 49990$, $m = 22$); \label{svm_dataset_5}
\item real sim ($n = 72309$, $m = 20958$); \label{svm_dataset_6}
\item webspam ($n = 350000$, $m = 254$); \label{svm_dataset_7}
\item covtype ($n = 581012$, $m = 54$). \label{svm_dataset_8}
\end{enumerate}

For each data set, we used as starting point a vector made of all zeros except for two randomly chosen variables
that were set strictly between the lower and the upper bound.
In AC2CD, the index $j(k)$ and the stepsize were computed as described before for the Chebyshev center problem.
In RCD, the working set was randomly chosen at each inner iteration by the same $\mathcal O(1)$ procedure used for the Chebyshev center problem.

The results for each algorithm are reported in Table~\ref{tab:svm} in terms of final objective value, number of (outer) iterations and CPU time in minutes.
On the first six data sets, which have less than $10^5$ samples, we see that MVP has the lowest CPU time,
but AC2CD still outperforms RCD. In particular, on these problems AC2CD is on average faster than RCD by a factor of almost $5$ in CPU time
and by a factor larger than $11$ in the number of outer iterations.
On the two largest data sets, having more than $10^5$ samples, AC2CD achieves the best performances.
In particular, considering the CPU time, on data set~\ref{svm_dataset_7} AC2CD is
more than $42$ times faster than RCD and more than $4$ times faster than MVP,
while on data set~\ref{svm_dataset_8} AC2CD is about $15$ times faster than RCD and about $13$ times faster than MVP.

{\setlength{\tabcolsep}{0.28em}
\begin{table}
\centering
\caption{Results of AC2CD, RCD and MVP on linear SVM training problems.
For each algorithm, the first column indicates the final objective value, the second column indicates the number of (outer) iterations
and the third column indicates the CPU time in minutes.}
{\begin{tabular}{ c | c c c | c c c | c c c}
\hline
\multirow{3}*{\minitab[c]{\textbf{Data} \\ \textbf{set}}}
& \multicolumn{3}{c|}{\textbf{AC2CD}} & \multicolumn{3}{c|}{\textbf{RCD}}  & \multicolumn{3}{c}{\textbf{MVP}} \bigstrut[t] \\
& \multirow{2}*{Obj} & Outer & Time   & \multirow{2}*{Obj} & Outer & Time  & \multirow{2}*{Obj} & \multirow{2}*{Iter} & Time \bigstrut[t] \\
&                    & iter  & (min)  &                    & iter  & (min) &                    &                     & (min) \bigstrut[b] \\
\hline
\ref{svm_dataset_1} & $-0.67$ & $35$ & $4.67$ & $-0.67$ & $97$ & $9.54$ & $-0.67$ & $3199$ & $1.60$ \bigstrut[t] \\
\ref{svm_dataset_2} & $-1745.37$ & $29$ & $25.79$ & $-1745.36$ & $101$ & $65.45$ & $-1745.36$ & $6642$ & $0.87$ \\
\ref{svm_dataset_3} & $-11432.18$ & $392$ & $3.63$ & $-11432.17$ & $4961$ & $27.19$ & $-11432.17$ & $76973$ & $1.44$ \\
\ref{svm_dataset_4} & $-1486.57$ & $491$ & $13.26$ & $-1486.57$ & $7398$ & $31.99$ & $-1486.57$ & $40957$ & $0.92$ \\
\ref{svm_dataset_5} & $-8589.43$ & $101$ & $0.45$ & $-8589.42$ & $2623$ & $4.35$ & $-8589.42$ & $17013$ & $0.36$ \\
\ref{svm_dataset_6} & $-5344.92$ & $39$ & $86.01$ & $-5344.91$ & $284$ & $436.74$ & $-5344.91$ & $21505$ & $4.63$ \\
\ref{svm_dataset_7} & $-69448.61$ & $60$ & $9.79$ & $-69448.55$ & $2298$ & $414.86$ & $-69448.55$ & $62436$ & $41.59$ \\
\ref{svm_dataset_8} & $-337942.88$ & $200$ & $25.23$ & $-337942.55$ & $3029$ & $378.62$ & $-337942.55$ & $936104$ & $328.81$ \bigstrut[b] \\
\hline
\end{tabular}}
\label{tab:svm}
\end{table}
}

\subsection{Problems with no bounds on the variables}
The last convex test problems in our experiments are of the form of~\eqref{unbound_prob}.
We considered the same class of objective functions used in~\cite{necoara:2017,xiao:2006}, that is,
\[
f_i(x) = \frac 1 2 a_i (x_i-c_i)^2 + \log\bigl(1+\exp(b_i(x_i-d_i))\bigr), \quad i = 1,\ldots,n,
\]
where $a_i > 0$, $i=1,\ldots,n$, and $b_i$, $c_i$, $d_i$ $\in \R$, $i=1,\ldots,n$.
It is possible to show that each $f_i$ is strongly convex with constant $a_i$
and has a Lipschitz continuous derivative with constant $L_i = a_i + (1/4)b_i^2$ (see~\cite{necoara:2017,xiao:2006}).

Six artificial problems were created. The first three have the following dimension:
\begin{enumerate*}[label=(\roman*)]
\item $n = 5000$, \label{no_bounds_n_5000_mata_1}
\item $n = 10000$, \label{no_bounds_n_10000_mata_1}
\item $n = 20000$, \label{no_bounds_n_20000_mata_1}
\end{enumerate*}
with $a_i$ randomly generated from a uniform distribution on $(0, 15)$ and $b_i$, $c_i$, $d_i$ randomly generated from a uniform distribution on $(-15, 15)$.
The last three problems have the same dimension as the previous ones, i.e.,
\begin{enumerate*}[label=(\roman*), start=4]
\item $n = 5000$, \label{no_bounds_n_5000_mata_2}
\item $n = 10000$, \label{no_bounds_n_10000_mata_2}
\item $n = 20000$, \label{no_bounds_n_20000_mata_2}
\end{enumerate*}
but with $a_i$ randomly generated from a uniform distribution on $(0, 2)$, $b_i$ randomly generated from a uniform distribution on $(-2, 2)$
and $c_i$, $d_i$ randomly generated from a uniform distribution on $(-10, 10)$.

In all of these problems, the zero vector was used as starting point.
In AC2CD, the index $j(k)$ was maintained fixed for all $k \ge 0$ and was chosen as an element of the set $\argmax_{i=1,\ldots,n} 1/L_i$,
while the stepsize was computed as described in Proposition~\ref{prop:lips_stepsize}, with $\gamma = 1/2$ and $\bar L_{i,j}$ replaced by $L_i + L_j$, $i,j = 1,\ldots,n$
(see Remark~\ref{rem:lips_stepsize}).

Our algorithm was compared with two versions of RCD proposed in~\cite{necoara:2017}, using blocks made of two variables
and different probability distributions (studied in~\cite{necoara:2017}), which are now described.
Denoting by $p_{ij}$ the probability to select a pair of distinct variable indices $(i,j)$
at any inner iteration of RCD, we first used a uniform distribution, i.e., all $p_{ij}$ have the same value,
and then we used probabilities that depend on the Lipschitz constants, i.e., each $p_{ij}$ is equal to
$(L_i^{-1} + L_j^{-1}) / \sum_{\substack{h,t = 1,\ldots,n \\ h \ne t}} (L_h^{-1} + L_t^{-1})$.
We name the two resulting algorithms RCD$_{\text{unif}}$ and RCD$_{\text{Lips}}$, respectively.
More in detail, to choose the working set at any inner iteration, in RCD$_{\text{unif}}$ we used the same $\mathcal O(1)$ procedure
described before for the Chebyshev center problem, while in RCD$_{\text{Lips}}$ we used the random generator
proposed by Nesterov in~\cite{nesterov:2012}, adjusted for our purposes.
It requires to randomly generate one real number from a uniform distribution on $(0,1)$ and perform $\mathcal O(\ln (\frac{n(n-1)}2))$ operations on some vectors
(whose preliminary definition has a cost that has not been included in the final statistics).

In Table~\ref{tab:no_bounds}, we report the results for each algorithm in terms of final objective value, number of outer iterations and CPU time in seconds.
We see that, on the first three problems, RCD$_{\text{unif}}$ achieves the best performances: in terms of CPU time it is faster than AC2CD
by a factor of about $1.5$, but AC2CD is almost $13$ times faster than RCD$_{\text{Lips}}$ on average.
On the last three problems, in terms of CPU time AC2CD outperforms both RCD$_{\text{unif}}$ and RCD$_{\text{Lips}}$
by an average factor of about $2$ and $15$, respectively.
Moreover, we observe that the number of outer iterations of AC2CD and RCD$_{\text{Lips}}$ is similar on all the considered problems,
but the amount of time needed to converge is remarkably different, with AC2CD being much faster.
This is due to the procedure used in RCD$_{\text{Lips}}$ to randomly generate, at each inner iteration, a pair of variable indices from a
Lipschitz-dependent probability distribution.

{\setlength{\tabcolsep}{0.32em}
\begin{table}
\centering
\caption{Results of AC2CD, RCD$_{\text{unif}}$ and RCD$_{\text{Lips}}$ on singly linearly constrained problems with no bounds on the variables.
For each algorithm, the first column indicates the final objective value, the second column indicates the number of outer iterations
and the third column indicates the CPU time in seconds.}
{\begin{tabular}{ c | c c c | c c c | c c c}
\hline
\multirow{3}*{\textbf{Problem}}
& \multicolumn{3}{c|}{\textbf{AC2CD}} & \multicolumn{3}{c|}{\textbf{RCD$_{\text{unif}}$}} & \multicolumn{3}{c}{\textbf{RCD$_{\text{Lips}}$}} \bigstrut[t] \\
& \multirow{2}*{Obj}  & Outer & Time  & \multirow{2}*{Obj}  & Outer & Time                & \multirow{2}*{Obj} & Outer & Time \bigstrut[t] \\
&                     & iter  & (s)   &                     & iter  & (s)                 &                    & iter  & (s) \bigstrut[b]  \\
\hline
\ref{no_bounds_n_5000_mata_1} & $166966.56$ & $40614$ & $87.42$ & $166966.73$ & $29625$ & $58.64$ & $166966.73$ & $40115$ & $1001.08$ \bigstrut[t] \\
\ref{no_bounds_n_10000_mata_1} & $336590.16$ & $7738$ & $33.61$ & $336590.50$ & $5969$ & $23.84$ & $336590.50$ & $7251$ & $414.50$ \\
\ref{no_bounds_n_20000_mata_1} & $671944.75$ & $46216$ & $387.71$ & $671945.42$ & $34181$ & $275.61$ & $671945.42$ & $43710$ & $5711.38$ \\
\ref{no_bounds_n_5000_mata_2} & $14510.90$ & $96$ & $0.19$ & $14510.91$ & $180$ & $0.37$ & $14510.92$ & $101$ & $2.57$ \\
\ref{no_bounds_n_10000_mata_2} & $30792.48$ & $484$ & $1.94$ & $30792.51$ & $1011$ & $4.11$ & $30792.51$ & $481$ & $27.76$ \\
\ref{no_bounds_n_20000_mata_2} & $58864.11$ & $882$ & $6.69$ & $58864.16$ & $1507$ & $12.31$ & $58864.16$ & $888$ & $117.41$ \bigstrut[b] \\
\hline
\end{tabular}}
\label{tab:no_bounds}
\end{table}
}

\subsection{Non-convex problems}\label{subsec:non_convex}
To test how AC2CD works when the objective function is non-convex, we finally considered problems of the form of~\eqref{non_convex_prob}.
Each problem was created by the following procedure: first the elements of $Q$ and those of $q$ were randomly generated from
a standard normal distribution and from a uniform distribution on $(0,1)$, respectively;
then the diagonal elements of $D$ were set to $1$, except for a prefixed number of them that were randomly chosen
and set to negative values randomly generated from a uniform distribution on $(-1,0)$.

More in detail, we generated three problems by fixing $n = m = 7000$ and
considering a number of negative diagonal elements of $D$ equal to $0.35m$, $0.5m$ and $0.65m$, respectively.
The three problems are summarized below:
\begin{enumerate}[label=(\roman*), leftmargin=*, itemsep=-1ex]
\item $n = m = 7000$, $n_{\text{neg}} = 2450$, $n_{\text{pos}} = 4550$,
    $\lambda_{\text{min}} = -9.19 \cdot 10^3$, $\lambda_{\text{max}} = 2.18 \cdot 10^4$;\label{non_convex_n_7000_neg_2450}
\item $n = m = 7000$, $n_{\text{neg}} = 3500$, $n_{\text{pos}} = 3500$,
    $\lambda_{\text{min}} = -1.13 \cdot 10^4$, $\lambda_{\text{max}} = 1.90 \cdot 10^4$;\label{non_convex_n_7000_neg_3500}
\item $n = m = 7000$, $n_{\text{neg}} = 4550$, $n_{\text{pos}} = 2450$,
    $\lambda_{\text{min}} = -1.30 \cdot 10^4$, $\lambda_{\text{max}} = 1.60 \cdot 10^4$;\label{non_convex_n_7000_neg_4550}
\end{enumerate}
where $n_{\text{neg}}$ denotes the number of negative eigenvalues of $Q^T D Q$,
$n_{\text{pos}}$ denotes the number of positive eigenvalues of $Q^T D Q$,
$\lambda_{\text{min}}$ denotes the smallest eigenvalue of $Q^T D Q$
and $\lambda_{\text{max}}$ denotes the largest eigenvalue of $Q^T D Q$.
Since in the non-convex case the final objective value found by an algorithm can depend on the starting point,
for each problem we considered $10$ different starting points, randomly chosen among the vertices of the unit simplex.

The procedures used to compute the stepsize and the index $j(k)$ in AC2CD, and that used to choose the working set in RCD,
were the same described before for the Chebyshev center problem.

In Table~\ref{tab:non_convex}, we report the results for each algorithm in terms of final objective value, number of (outer) iterations and CPU time in seconds.
For each problem, we use the acronyms \textit{sp 1, \ldots, sp 10} to distinguish the results obtained with the $10$ considered starting points,
while \textit{avg} indicates the results averaged over the $10$ runs.
We first observe that, for problems~\ref{non_convex_n_7000_neg_3500} and~\ref{non_convex_n_7000_neg_4550}, the final objective values found by AC2CD
are on average smaller than those found by RCD and MVP, while we have the opposite situation for problem~\ref{non_convex_n_7000_neg_2450}.
We also see that there is a notable difference between AC2CD and RCD in both CPU time and the number of outer iterations,
especially on problem~\ref{non_convex_n_7000_neg_4550}.
Finally, in comparison with MVP, we note that AC2CD is on average faster on problems~\ref{non_convex_n_7000_neg_2450} and~\ref{non_convex_n_7000_neg_3500},
but it is slower on problem~\ref{non_convex_n_7000_neg_4550}.

{\setlength{\tabcolsep}{0.33em}
\begin{table}
\centering
\caption{Results of AC2CD, RCD and MVP on non-convex problems.
For each algorithm, the first column indicates the final objective value, the second column indicates the number of (outer) iterations
and the third column indicates the CPU time in seconds.}
{\begin{tabular}{ D{-}{\,-\,}{5} | c c c | c c c | c c c}
\hline
\multicolumn{1}{c|}{\multirow{3}*{\textbf{Problem}}}
& \multicolumn{3}{c|}{\textbf{AC2CD}} & \multicolumn{3}{c|}{\textbf{RCD}}  & \multicolumn{3}{c}{\textbf{MVP}} \bigstrut[t] \\
& \multirow{2}*{Obj} & Outer & Time   & \multirow{2}*{Obj} & Outer & Time  & \multirow{2}*{Obj} & \multirow{2}*{Iter} & Time \bigstrut[t] \\
&                    & iter  & (s)    &                    & iter  & (s)   &                    &                     & (s) \bigstrut[b] \\
\hline
\text{\ref{non_convex_n_7000_neg_2450}} - \text{sp 1} & $-3.10$ & $925$ & $110.90$ & $-3.07$ & $3473$ & $135.45$ & $-3.12$ & $51768$ & $715.26$ \bigstrut[t] \\
\text{\ref{non_convex_n_7000_neg_2450}} - \text{sp 2} & $-3.11$ & $493$ & $57.91$ & $-3.06$ & $4670$ & $172.26$ & $-3.00$ & $9496$ & $131.18$ \\
\text{\ref{non_convex_n_7000_neg_2450}} - \text{sp 3} & $-3.09$ & $582$ & $67.56$ & $-3.14$ & $6585$ & $231.34$ & $-3.05$ & $24364$ & $336.40$ \\
\text{\ref{non_convex_n_7000_neg_2450}} - \text{sp 4} & $-2.99$ & $278$ & $33.01$ & $-3.10$ & $2896$ & $108.96$ & $-3.06$ & $30712$ & $424.11$ \\
\text{\ref{non_convex_n_7000_neg_2450}} - \text{sp 5} & $-3.07$ & $507$ & $59.29$ & $-3.13$ & $3139$ & $113.96$ & $-3.06$ & $31679$ & $437.19$ \\
\text{\ref{non_convex_n_7000_neg_2450}} - \text{sp 6} & $-3.03$ & $411$ & $48.16$ & $-3.13$ & $3563$ & $128.45$ & $-3.06$ & $21713$ & $300.07$ \\
\text{\ref{non_convex_n_7000_neg_2450}} - \text{sp 7} & $-2.97$ & $472$ & $55.40$ & $-3.14$ & $4325$ & $153.75$ & $-3.22$ & $47542$ & $657.03$ \\
\text{\ref{non_convex_n_7000_neg_2450}} - \text{sp 8} & $-3.14$ & $650$ & $75.42$ & $-3.04$ & $1776$ & $67.81$ & $-3.22$ & $32163$ & $444.68$ \\
\text{\ref{non_convex_n_7000_neg_2450}} - \text{sp 9} & $-3.05$ & $444$ & $51.85$ & $-3.14$ & $2606$ & $94.93$ & $-3.10$ & $30057$ & $415.61$ \\
\text{\ref{non_convex_n_7000_neg_2450}} - \text{sp 10} & $-3.09$ & $474$ & $55.45$ & $-3.09$ & $4199$ & $149.46$ & $-3.10$ & $19097$ & $263.63$ \bigstrut[b] \\
\hline
\text{\ref{non_convex_n_7000_neg_2450}} - \text{avg} & $-3.06$ & $523.60$ & $61.49$ & $-3.10$ & $3723.20$ & $135.63$ & $-3.10$ & $29859.10$ & $412.52$ \bigstrut[t] \bigstrut[b] \\
\hhline{=|=:=:=|=:=:=|=:=:=|}
\text{\ref{non_convex_n_7000_neg_3500}} - \text{sp 1} & $-7.91$ & $996$ & $94.34$ & $-8.19$ & $11696$ & $126.07$ & $-7.46$ & $2597$ & $35.35$ \bigstrut[t] \\
\text{\ref{non_convex_n_7000_neg_3500}} - \text{sp 2} & $-7.80$ & $738$ & $69.55$ & $-7.61$ & $9323$ & $100.88$ & $-7.69$ & $4206$ & $57.38$ \\
\text{\ref{non_convex_n_7000_neg_3500}} - \text{sp 3} & $-7.37$ & $202$ & $19.27$ & $-7.44$ & $3533$ & $40.79$ & $-7.87$ & $2698$ & $36.58$ \\
\text{\ref{non_convex_n_7000_neg_3500}} - \text{sp 4} & $-7.69$ & $292$ & $27.83$ & $-7.56$ & $16032$ & $168.96$ & $-7.41$ & $5122$ & $69.61$ \\
\text{\ref{non_convex_n_7000_neg_3500}} - \text{sp 5} & $-7.71$ & $445$ & $42.15$ & $-7.69$ & $12016$ & $126.37$ & $-7.53$ & $8379$ & $114.22$ \\
\text{\ref{non_convex_n_7000_neg_3500}} - \text{sp 6} & $-7.53$ & $512$ & $48.44$ & $-7.20$ & $5140$ & $61.16$ & $-7.71$ & $3804$ & $51.84$ \\
\text{\ref{non_convex_n_7000_neg_3500}} - \text{sp 7} & $-7.87$ & $390$ & $36.97$ & $-8.07$ & $23649$ & $252.05$ & $-7.71$ & $3015$ & $40.83$ \\
\text{\ref{non_convex_n_7000_neg_3500}} - \text{sp 8} & $-7.63$ & $284$ & $27.11$ & $-7.46$ & $6617$ & $71.65$ & $-7.56$ & $11563$ & $157.08$ \\
\text{\ref{non_convex_n_7000_neg_3500}} - \text{sp 9} & $-8.05$ & $389$ & $36.78$ & $-7.65$ & $16883$ & $187.66$ & $-7.53$ & $5112$ & $69.51$ \\
\text{\ref{non_convex_n_7000_neg_3500}} - \text{sp 10} & $-7.58$ & $764$ & $72.03$ & $-7.80$ & $5020$ & $59.14$ & $-7.61$ & $5694$ & $77.43$ \bigstrut[b] \\
\hline
\text{\ref{non_convex_n_7000_neg_3500}} - \text{avg} & $-7.72$ & $501.20$ & $47.45$ & $-7.67$ & $10990.90$ & $119.47$ & $-7.61$ & $5219.00$ & $70.98$ \bigstrut[t] \bigstrut[b] \\
\hhline{=|=:=:=|=:=:=|=:=:=|}
\text{\ref{non_convex_n_7000_neg_4550}} - \text{sp 1} & $-56.61$ & $32$ & $2.54$ & $-56.77$ & $48600$ & $65.35$ & $-44.98$ & $44$ & $0.59$ \bigstrut[t] \\
\text{\ref{non_convex_n_7000_neg_4550}} - \text{sp 2} & $-51.07$ & $37$ & $3.02$ & $-50.38$ & $12031$ & $26.98$ & $-54.02$ & $19$ & $0.25$ \\
\text{\ref{non_convex_n_7000_neg_4550}} - \text{sp 3} & $-72.15$ & $51$ & $4.30$ & $-48.97$ & $58935$ & $125.35$ & $-55.15$ & $59$ & $0.78$ \\
\text{\ref{non_convex_n_7000_neg_4550}} - \text{sp 4} & $-53.08$ & $32$ & $2.57$ & $-54.34$ & $19085$ & $41.03$ & $-48.57$ & $27$ & $0.37$ \\
\text{\ref{non_convex_n_7000_neg_4550}} - \text{sp 5} & $-53.35$ & $138$ & $11.63$ & $-56.94$ & $19784$ & $35.11$ & $-44.35$ & $29$ & $0.39$ \\
\text{\ref{non_convex_n_7000_neg_4550}} - \text{sp 6} & $-41.55$ & $120$ & $10.15$ & $-47.77$ & $48053$ & $58.37$ & $-58.46$ & $100$ & $1.33$ \\
\text{\ref{non_convex_n_7000_neg_4550}} - \text{sp 7} & $-70.33$ & $80$ & $6.92$ & $-71.63$ & $17600$ & $68.96$ & $-71.63$ & $13$ & $0.18$ \\
\text{\ref{non_convex_n_7000_neg_4550}} - \text{sp 8} & $-54.30$ & $74$ & $6.30$ & $-52.52$ & $26160$ & $27.18$ & $-54.20$ & $20$ & $0.27$ \\
\text{\ref{non_convex_n_7000_neg_4550}} - \text{sp 9} & $-67.51$ & $16$ & $1.23$ & $-71.77$ & $15153$ & $36.87$ & $-51.24$ & $73$ & $0.98$ \\
\text{\ref{non_convex_n_7000_neg_4550}} - \text{sp 10} & $-61.52$ & $60$ & $4.79$ & $-56.74$ & $8267$ & $9.33$ & $-47.67$ & $57$ & $0.75$ \bigstrut[b] \\
\hline
\text{\ref{non_convex_n_7000_neg_4550}} - \text{avg} & $-58.15$ & $64.00$ & $5.34$ & $-56.78$ & $27366.80$ & $49.45$ & $-53.03$ & $44.10$ & $0.59$ \bigstrut[t] \bigstrut[b] \\
\hline
\end{tabular}}
\label{tab:non_convex}
\end{table}
}

\section{Conclusions}\label{sec:conc}
In this paper, a block coordinate descent method has been presented for minimizing a continuously differentiable function
subject to one linear equality constraint and simple bounds on the variables.
In the proposed method, the working set is chosen according to an almost cyclic strategy that does not use first-order information.
So, the whole gradient of the objective function does not need to be computed during the algorithm,
leading to high efficiency when the problem dimension is large and the partial derivatives of the objective function are cheap.
Global convergence to stationary points has been established under an appropriate assumption on the level set,
and linear convergence rate has been proved under standard additional assumptions.
Promising numerical results have been obtained on different classes of test problems.

There are a number of open questions that indicate directions in which this work can be extended and that can represent challenging tasks for future research.
First, it would be worth investigating if, by suitably modifying the working set selection rule or adding conditions to the stepsize,
global convergence can be obtained without Assumption~\ref{assumpt:l0_int_point}.
Other interesting questions would be how to generalize the proposed method to problems with more than one linear equality constraint,
and how to adjust our approach to realize a parallel algorithmic scheme (for example, by a Jacobi-type approach).
We wish to report further results in the future.

\bibliography{cristofari2018}

\end{document}